\documentclass[12pt]{article}

\usepackage[table]{xcolor}
\usepackage{amsmath, amsthm}
\usepackage{amsfonts,amssymb}
\usepackage{hyperref}
\usepackage{authblk}
\usepackage[all]{xy}
\usepackage{enumitem}
\usepackage{stmaryrd}
\usepackage[makeroom]{cancel}
\usepackage{bbm}
\usepackage{fancyvrb}

\usepackage[linesnumbered,lined,boxed]{algorithm2e}

\usepackage{xstring} 

\usepackage{tikz}
\usetikzlibrary{intersections}

\newif\ifextmath
\extmathfalse

\renewcommand{\le}{\leqslant}
\renewcommand{\ge}{\geqslant}

\newcommand{\N}{\mathbb{N}}
\newcommand{\Z}{\mathbb{Z}}
\newcommand{\Q}{\mathbb{Q}}

\newcommand{\C}{\mathbb{C}}
\newcommand{\F}{\mathbb{F}}

\newcommand{\p}{\mathfrak{p}}

\newcommand{\PP}{\mathbb{P}}
\newcommand{\A}{\mathbb{A}}
\renewcommand{\L}{\mathcal{L}}
\renewcommand{\O}{\mathcal{O}}

\newcommand{\Gal}{\operatorname{Gal}}

\newcommand{\GL}{\operatorname{GL}}
\newcommand{\SL}{\operatorname{SL}}
\newcommand{\GSp}{\operatorname{GSp}}
\newcommand{\PGL}{\operatorname{PGL}}

\newcommand{\Fl}{{\F_\ell}}
\newcommand{\Flx}{{\F_\ell^\times}}

\newcommand{\Ker}{\operatorname{Ker}}
\renewcommand{\Im}{\operatorname{Im}}
\newcommand{\Eqn}{\operatorname{Eqn}}
\newcommand{\codim}{\operatorname{codim}}
\newcommand{\rank}{\operatorname{rk}}

\newcommand{\Res}{\operatorname{Res}}

\newcommand{\Hom}{\operatorname{Hom}}

\newcommand{\lcm}{\operatorname{lcm}}

\newcommand{\charf}{\mathbbm{1}}

\newcommand{\Pic}{\operatorname{Pic}}
\newcommand{\Eff}{\operatorname{Eff}}

\SetKw{True}{True}
\SetKw{False}{False}
\newcommand{\la}{\leftarrow}
\SetKw{FAIL}{FAIL}
\SetKw{Goto}{go to line}

\newcommand{\DivAdd}{\operatorname{DivAdd}}
\newcommand{\DivSub}{\operatorname{DivSub}}

\numberwithin{equation}{section}

\newtheorem{thm}[equation]{Theorem}
\newtheorem{lem}[equation]{Lemma}
\newtheorem{pro}[equation]{Proposition}

\theoremstyle{definition}
\newtheorem{de}[equation]{Definition}
\newtheorem{rk}[equation]{Remark}

\newcounter{para}[subsection]
\renewcommand{\thepara}{\thesubsection.\arabic{para}}
\newcommand\mypara{\par\refstepcounter{para}\noindent\textbf{\thepara.}\space}

\makeatletter
\newcommand{\subjclass}[2][2010]{%
  \let\@oldtitle\@title%
  \gdef\@title{\@oldtitle\footnotetext{#1 \emph{Mathematics subject classification:} #2}}%
}
\newcommand{\keywords}[1]{%
  \let\@@oldtitle\@title%
  \gdef\@title{\@@oldtitle\footnotetext{\emph{Key words and phrases.} #1.}}%
}
\let\c@table\c@equation
\let\c@figure\c@equation
\makeatother
\makeatother

\title{Hensel-lifting torsion points on Jacobians \\ and Galois representations}
\author{Nicolas Mascot\thanks{\href{mailto:nm116@aub.edu.lb}{nm116@aub.edu.lb}}}
\affil{\scriptsize{AUB, Beirut, Lebanon}}

\VerbatimFootnotes

\begin{document}

\maketitle

\begin{abstract}
Let~$\rho$ be a mod~$\ell$ Galois representation. We show how to compute~$\rho$ explicitly, given the characteristic polynomial of the image of the Frobenius at one prime~$p$ and a curve~$C$ whose Jacobian contains~$\rho$ in its~$\ell$-torsion. The main ingredient is a method to~$p$-adically lift torsion points on a Jacobian in the framework of Makdisi's algorithms.
\end{abstract}

\renewcommand{\abstractname}{Acknowledgements}
\begin{abstract}
The author thanks Aurel Page for his implementation of the computation of the Howell form of kernels in~\cite{gp} and his help with the proof of theorem~\ref{thm:Howell} and the formulation of the analysis of the performance of algorithm~\ref{alg:Vbasis}, Aurel Page and Bill Allombert for their help with~\cite{gp}'s C~language library, and Kamal Khuri-Makdisi for some useful conversations.

The computer algebra packages used for the computations presented in this paper were almost exclusively~\cite{gp}, plus a little bit of~\cite{Magma} and~\cite{Gap}.
\end{abstract}


\bigskip

\section{Introduction}

\subsection{Statement of the problem and notations}

The notations set in this section will be used throughout the rest of this article.

\bigskip

Let~$\ell \in \N$ be a prime number, let~$d \in \N$, and let~$\rho : \Gal(\overline \Q/\Q) \longrightarrow \GL_d(\Fl)$ be a mod~$\ell$ Galois representation of degree~$d$. We would like to compute~$\rho$, that is to say determine explicitly a polynomial~$F(x) \in \Q[x]$ whose splitting field is the Galois number field corresponding to the kernel of~$\rho$, as well as an explicit indexation of the roots of~$F(x)$ (in some large enough field where~$F(x)$ splits completely) by the points of~$\Fl^d \setminus \{ 0 \}$ such that the Galois action on these roots corresponds to the action of~$\rho$ on~$\Fl^d \setminus \{ 0 \}$. Indeed, given such data and a prime~$p \in \N$ at which~$\rho$ is unramified, it is not very difficult to determine the image by~$\rho$ of the Frobenius at~$p$ up to conjugacy in~$\Im \rho$, even for very large~$p$, thanks to the technique presented in~\cite{Dok}.

Suppose that we have an explicit model (for instance, equations) for a proper, non-singular, geometrically connected curve~$C$ of genus~$g$ defined over~$\Q$, such that there exists a~$d$-dimensional~$\F_\ell$-subspace~$T_\rho$ of the~$\ell$-torsion~$J[\ell]$ of the Jacobian~$J$ of~$C$ such that~$T_\rho$ affords the Galois representation~$\rho$. We wish to use the knowledge of the curve~$C$ to compute~$\rho$. For example, we showed in~\cite{algo} and~\cite{companion} how to compute the Galois representations attached to classical modular forms, by taking~$C$ to be a modular curve.

In the case when~$C$ is a modular curve, we may use the Hecke action in order to single~$T_\rho$ out of~$J[\ell]$ (cf.~\cite{algo} for details), but we have no such tool at our disposal when~$C$ is a general curve. Therefore, we instead assume that there is (at least) one prime~$p \in \N$ at which~$\rho$ is unramified such that we know explicitly the characteristic polynomial~$\chi_\rho(x) \in \Fl[x]$ of the image by~$\rho$ of the Frobenius~$\Phi$ at~$p$. We actually make the requirement that~$p$ is distinct from~$\ell$ and that our model of~$C$ has good reduction at~$p$ (which implies that~$\rho$ is unramified at~$p$ by N\'eron-Ogg-Shafarevich). Furthermore, we suppose that the local~$L$-factor~$L_p(x) \in \Z[x]$ of~$C$ at~$p$ is known explicitly, and that~$\chi_\rho(x)$, which necessarily divides~$L_p(x)$ mod~$\ell$, is actually coprime mod~$\ell$ to its cofactor~$L_p(x) / \chi_\rho(x)$. This last requirement ensures that the knowledge of~$\chi_\rho(x)$ determines the subspace~$T_\rho$ of~$J[\ell]$ uniquely.

 The goal of this article is to present an algorithm that, given a model for~$C$, the prime~$p \in \N$ and the characteristic polynomial~$\chi_\rho(x)$, computes~$\rho$ in the above sense.

\begin{rk}
Since we have an explicit model of~$C$, the local~$L$-factor~$L_p(x)$ may in principle be recovered as the numerator of the Hasse-Weil zeta function of the reduction mod~$p$ of~$C$ by point-counting techniques. Here we assume that we have somehow been able to determine it. In practice, this often means that~$p$ cannot be too large.
\end{rk}

\begin{rk}
We do not exclude the case where~$T_\rho = J[\ell]$, i.e. where the representation we want to compute is that afforded by the whole~$\ell$-torsion of~$J$. In this case, the output of our algorithm may be called an~$\ell$-division polynomial of~$C$. Thanks to~\cite{Dok}, we can then compute the local factor~$L_{p'}(x) \bmod \ell$ in time polynomial in~$\log p'$ for primes~$p' \in \N$ of good reduction. By Chinese remainders, we can thus determine~$\#C(\F_{p'^n})$
, at least in theory. This approach results in a point-counting method whose complexity is as good as~\cite{Gaudry}, but which is more general since it is not limited to hyperelliptic curves.
\end{rk}

\begin{rk} Although we focus on the case of representations of~$\Gal(\overline \Q/\Q)$ for the simplicity of exposition, it is completely straightforward to generalize the techniques presented in this article to the case of mod~$\ell$ representations of~$\Gal(\overline k / k)$ where~$k$ is a number field, by taking~$C$ to be defined over~$k$, and replacing~$p$ by a finite prime~$\p$ of~$k$ whose residual characteristic is distinct from~$\ell$ and at which~$C$ has good reduction.
\end{rk}

\begin{rk}
The case of a representation with values in~$\GL_d(\F_\lambda)$, where~$\F_\lambda$ is a finite extension of~$\Fl$, can also be treated by our methods (provided of course that the curve~$C$ is known), since restriction of scalars and the flexibility of~\cite{Dok} allows us to treat it as a representation with values in a subgroup of~$\GL_{d [\F_\lambda : \Fl]}(\F_\ell)$.
\end{rk}

\subsection{Outline}

Our strategy to compute~$\rho$ may be summarized as follows:

\begin{enumerate}
\item Compute~$d$ points forming a basis of the image of~$T_\rho$ in~$J(\F_q)[\ell]$, where~$q=p^a$ with~$a \in \N$ large enough to split~$\rho$ (i.e. large enough that the points of the reduction of~$T_\rho$ mod~$p$ are defined over~$\F_q$),
\item Fix an integer~$e \in \N$, and lift these points to approximations with accuracy~$O(p^e)$ of~$\ell$-torsion points in~$J(\Q_q)$, that is to say to points of~$J(\Z_q/p^e)$, where~$\Q_q$ denotes the unramified extension of~$\Q_p$ whose residue field is~$\F_q$ and~$\Z_q$ is the ring of integers of~$\Q_q$,
\item Compute all the~$\Fl$-linear combinations of these~$d$ points, thus obtaining a model of~$T_\rho$ over~$\Z_q/p^e$,
\item Evaluate a rational map~$\alpha : J \dashrightarrow \A^1$ defined over~$\Q$ at these~$\ell^d$ points, and form the monic polynomial~$F(x)$ whose roots are these values,
\item \label{Outline:check_injective} Check that these values are distinct, else take another map from~$J$ to~$\A^1$ and try again,
\item \label{Outline:ident} Identify~$F(x)$ as a polynomial with coefficients in~$\Q$, or, if this fails, start over with a larger value of~$e$.
\end{enumerate}

Indeed, for each nonzero~$t \in T_\rho$,~$\alpha(t)$ lies in the field of definition~$\Q(t)$ of~$t$, so the splitting field of~$F(x)$ is a subfield of the field corresponding to the kernel of~$\rho$, and agrees with it with high probability since most elements of a number field are primitive elements. More precisely, a sufficient condition for the Galois action on the roots of~$F(x)$ to reflect faithfully that on the points of~$T_\rho$ is that~$\alpha$ be injective on~$T_\rho$, which is what we check at step~\ref{Outline:check_injective}.

\begin{rk}
Because of the identification performed at step~\ref{Outline:ident}, the output of this algorithm is not guaranteed to be correct, and should be certified by methods such as~\cite{certif}.
\end{rk}

Throughout this algorithm, we use Kamal Khuri-Makdisi's methods~\cite{Mak1},~\cite{Mak2} to compute in~$J$. We give a brief presentation of these techniques in section~\ref{sect:Mak}, including a recipe to construct rational maps from~$J$ to~$\A^1$ and comments on the implementation over a~$p$-adic field.

The core of this article is devoted to the presentation of a technique to lift torsion points from~$\F_q$ to~$\Q_q$ (up to some arbitrary~$p$-adic accuracy) in Makdisi's framework. After differential calculus preliminaries in section \ref{sect:diff}, we show how to lift a point from~$J(\F_q)$ to~$J(\Z_q/p^e)$ for any~$e \in \N$ in section \ref{sect:lift}, and then how to ensure that the lift of an~$\ell$-torsion point over~$\F_q$ remains~$\ell$-torsion over~$\Z_q/p^e$ in section \ref{sect:torslift}.

Finally, we explain in detail how to use these tools to compute Galois representations in section \ref{sect:galrep}, and we give explicit examples in section \ref{sect:examples}.

\begin{rk}
As we will comment on in the next sections, in Makdisi's method, points on the Jacobian are represented by matrices of a certain size. The crucial point is that most matrices of this given size do not represent any point on the Jacobian; thus lifting a point from~$J(\F_q)$ to~$J(\Z_q/p^e)$ is a non-trivial process. In~\cite{algo}, we used Makdisi's methods over~$\C$, and we noticed that the accuracy deteriorated little-by-little after each operation in~$J$, in that the matrix crept away from the locus corresponding to actual points on the Jacobian. This is the reason that initially prompted us to look for a method to ``fix'' representations of points on~$J$ in Makdisi's method given with poor accuracy. We then realized that performing this ``fix''~$p$-adically led to an algorithm to compute Galois representations. The methods presented in this article ought to also be able to ``fix'' complex approximations of points in Makidisi's method, although we have not actually tried them this way.
\end{rk}

\section{Computing in the Jacobian}\label{sect:Mak}

\subsection{Review of Makdisi's algorithms}

We start by recalling how Makdisi's algorithms work; indeed we will reuse some of their ideas in this article. None of the material present in this subsection is original work.

\begin{rk}Makdisi's algorithms are very flexible and can be implemented in many different ways, cf.~\cite{Mak1} and~\cite{Bruin} for more details. The version we give here is the one that is best suited to our purpose. Connoisseurs will recognize the medium model in representation B$_0$.
\end{rk}

\subsubsection{Representation of spaces of functions on the curve}

As in the previous section, we consider a nice algebraic curve~$C$ of genus~$g$. However, the ground field~$k$ need not be~$\Q$ anymore, but could be any perfect field over which one may perform linear algebra effectively. Given a divisor~$D$ on~$C$, we will denote the corresponding Riemann-Roch space by
\[ \L(D) =\{ f \in  k(C)^\times \ \vert \ (f) + D \geqslant 0 \} \cup \{ 0 \}. \]

We begin by picking an effective divisor~$D_0$ defined over~$k$ on~$C$, of degree~$d_0 \geqslant 2g+1$, and we define 
\[ V_n = \L(nD_0) \]
for all~$n \in \N$. The space~$V_2$ will play a particular r\^ole, so we set
\[ V = V_2 \]
for brevity.

Let~$n_Z \ge 5d_0+1$ be an integer. We will assume that~$k$ is large enough that we can fulfill the following condition:
\begin{equation}
\text{There exists an effective divisor } Z \text{ formed of } n_Z \text{ \emph{distinct} points } P_i \text{ in } C(k) \setminus \operatorname{supp} D_0. \label{nZ_ratpts}
\end{equation}

\begin{rk}
For efficiency reasons, we should take~$d_0$ as close to~$2g+1$ as possible and $n_Z$ as close to~$5d_0+1$ as possible (bearing in mind constraints such as those described in remark~\ref{rk:rationality_D0_Ei} and subsection~\ref{sect:Frob}). We will therefore assume that~$d_0 = O(g)$ and that~$n_Z=O(g)$ for complexity analyses. 
\end{rk}

All the computations performed by Makdisi's algorithms take place in the space~$V_5$ (and so do the ones presented in the article). We represent an element~$v \in V_5$ by a column vector in~$k^{n_Z}$ containing its values~$v(P_i)$ at the~$n_Z$  points~$P_i \in C(k)$. Since such a function has degree at most~$5 d_0$, this representation is faithful, and defines an embedding of~$V_5$ into~$k^{n_Z}$. 

In this representation, addition and multiplication of functions is done point-wise, which will simplify our task in the next sections. Given two column vectors~$c, c' \in k^{n_Z}$, we denote by~$c \odot c'$ their point-wise product, that is to say the column vector whose~$i$-th entry is~$c_i c'_i$ for all~$i$. Thus if~$c$ and~$c'$ represent functions on~$C$,~$c \odot c'$ represents their product (provided that all these functions lie in~$V_5$).

We will work with subspaces of~$V_5$. Such a subspace~$S$ will be represented by a matrix
\[ M = \begin{pmatrix} s_1(P_1) & \cdots & s_{\dim S} (P_1) \\ \vdots & & \vdots \\ s_1(P_{n_Z}) & \cdots & s_{\dim S} (P_{n_Z}) \end{pmatrix}\]
of size~$n_Z \times \dim S$ whose~$i,j$-entry is~$s_j(P_i)$, where the~$s_j \in S$ form a~$k$-basis of~$S$ and the~$P_i$ are as above. Thus the columns of~$M$ represent a basis of~$S$ in the above sense. Since~$S$ has many~$k$-bases, this representation is of course not unique. We will call such a matrix a matrix \emph{representing}~$S$, or a matrix \emph{for}~$S$.

Given a column~$c \in k^{n_Z}$ and a matrix~$M$ with~$n_Z$ rows, we denote by~$c \odot M$ the matrix whose column are the~$c \odot M_j$, where the~$M_j$ are the columns of~$M$. Thus if~$c$ represents a function~$f$ and~$M$ represents a subspace~$S$ of~$V_5$, then~$f \odot M$ represents the subspace~$fS = \{ f s, \ s \in S\}$ (provided that this is still a subspace of~$V_5$).

Later on, we will frequently need to compute a matrix~$K_S$ of size~$(n_Z-\dim S) \times n_Z$ whose rows represent independent linear equations defining~$S$ as a subspace of~$k^{n_Z}$. Such a matrix can be obtained simply as the transpose of the kernel of the transpose of~$M$, and will be referred to as an \emph{equation matrix} for~$S$.\label{def_eqnmatrix}

Conversely, given a matrix~$A$ whose rows represent linear equations, we call a matrix whose columns represent a basis of the space of solutions of these linear equations a \emph{kernel matrix} of~$A$.

\begin{rk}\label{rk:padic_conventions}
In this article, we will actually work with~$k=\F_q$ or~$\Q_q$, where~$q = p^a$ is a prime power, and we impose that the~$n_Z$ points~$P_i \in C(k)$ remain distinct mod~$p$. The Hasse bounds show that this is always achievable, possibly at the price of enlarging~$a$, and also that~$q$ cannot be very small. We will also suppose that matrices representing a subspace $S$ of $V_5$, as well as kernel (resp. equation) matrices, have coefficients in the ring of integers~$\Z_q$ of~$\Q_q$, and that their columns (resp. rows) remain linearly independent when reduced mod~$p$. We will show how to preserve this condition throughout our computations, and more generally how to perform linear algebra over~$\Z_q$ without loss of~$p$-adic accuracy, in section \ref{sect:Howell} below.
\end{rk}

By Riemann-Roch, every point~$x \in J=\Pic^0(C)$ is of the form~$x=[D-D_0]$ for some (non-unique) effective divisor~$D$ of degree~$d_0$. Such a point will be represented by a matrix~$W_D$ for the~$d_0$-codimensional subspace~$\L(2D_0-D) \subsetneq V$
of~$V$. Thus~$x$ is represented by an effective divisor~$D$ such that~$[D-D_0]=x$, itself represented by the subspace~$\L(2D_0-D)$, itself represented by a matrix~$W_D$ of size~$n_Z \times d_W$, where \[d_W = d_0+1-g\]
is the dimension of such a subspace.
Since~$\deg(2D_0 - D) = d_0 \ge 2g+1$, the Riemann-Roch space~$\L(2D_0-D)$ is base-point-free, so this representation is faithful. However, it is obviously not unique.

In particular, the space~$V_1 = \L(D_0)$ represents the origin of~$J$, so we denote by~$W_0$ a matrix representing this space (thus $W_0$ means $W_D$ for $D=D_0$, not $D=0$).

We suppose that we can compute an explicit basis of~$V_1=\L(D_0)=W_0$. This presents various difficulties depending on what kind of model of~$C$ we have, cf. for instance~\cite{Hess}. Thanks to algorithm \ref{alg:divadd} introduced below, we can then compute matrices representing~$V_n$ for any~$n \in \N$. We initialize Makdisi's algorithms by precomputing matrices~$W_0$,~$M_V$, and~$M_{V_{3}}$ representing respectively~$V_1$,~$V=V_2$, and~$V_3$, as well as equation matrices~$K_V$ and~$K_{V_3}$ for~$V$ and~$V_3$. Although all the action will take place in~$V_5$, we do not need to compute a matrix (nor an equation matrix) for~$V_5$.

\begin{rk}
The dissension between the notations~$M_V$, $M_{V_3}$ (matrices for~$\L(2D_0)$ and ~$\L(3D_0)$) and~$W_D$ (matrix for the subspace~$\L(2D_0-D)$ representing~$[D-D_0] \in J$) is meant as a reminder that~$M_V$ and ~$M_{V_3}$ are ``structure constants'' describing the curve~$C$, whereas~$W_D$ represents a (variable) point on~$J$.
\end{rk}

\subsubsection{Arithmetic on divisors}

Before we present Makdisi's algorithms \emph{per se}, we show how we can add and subtract divisors at the level of Riemann-Roch spaces. Let us start with addition. We set the following notation: given two subspaces~$S, T \subset k^{n_Z}$, we let~$S \cdot T$ be the space spanned by the products~$st$ where~$s \in S$ and~$t \in T$.

\medskip

\begin{algorithm}[H]
\KwIn{Matrices~$M_A$ and~$M_B$ for the spaces~$\L(A)$ and~$\L(B)$, where~$A$,~$B$ are arbitrary divisors such that~$A,B,A+B \le 5D_0$ and~$\deg A, \deg B \ge d_0$.}
\KwOut{A matrix for the space~$\L(A+B)$.}
$d \la \dim \L(A) + \dim \L(B) +g-1$\tcp*{$d = \dim \L(A+B)$} \label{alg:divadd:dim}
$S \la \emptyset$\;
\Repeat{$\#S=d$}{
$s\la\#S$\;
\For{$j \la 1$ \KwTo~$d-s$}
{\label{alg:divadd:prods}
$a \la$ a random combination of columns of~$M_A$\;
$b \la$ a random combination of columns of~$M_B$\;
Append~$a \odot b$ to~$S$\;
}
$S \la$ a basis of the subspace of~$k^{n_Z}$ spanned by~$S$\; \label{alg:divadd:dep}
}
\Return the matrix whose columns are the elements of~$S$\;
\caption{DivAdd}
\label{alg:divadd}
\end{algorithm}

\begin{proof}
Makdisi proves in~\cite[lemma 2.2]{Mak1} that if~$A$ and~$B$ both have degree at least~$2g+1$, then the multiplication map
\[ \L(A) \otimes \L(B) \longrightarrow \L(A+B) \]
 is surjective; in other words, we have~$\L(A+B)=\L(A) \cdot \L(B)$. The dimension of this space is computed at line \ref{alg:divadd:dim} by Riemann-Roch. In each iteration of the \textbf{for} loop,~$a$ (resp.~$b$)  represents a random element of~$\L(A)$(resp.~$\L(B)$), so~$a \odot b$ represents an element of~$\L(A) \cdot \L(B) = \L(A+B)$. We keep forming such products until we generate~$\L(A+B)$.
 
The condition~$A$,~$B$,~$A+B \le 5D_0$ ensures that these spaces are subspaces of~$V_5$, so that their elements are represented faithfully by their values at the points~$P_i$.
\end{proof}

\begin{rk}
If~$k$ is a very small finite field, it can be advantageous to execute the \textbf{for} loop at line \ref{alg:divadd:prods} a little more than~$d-\#S$ times so as to ensure we get a generating set of~$\L(A+B)$ with high probability (A detailed analysis of the situation is given in~\cite{Bruin}). However, our assumption \eqref{nZ_ratpts} implies that~$k$ cannot be too small, so that in practice it is extremely rare that we need to execute the \textbf{repeat}...\textbf{until} loop more than once even if we take just~$d$ products~$a\odot b$.
\end{rk}

\begin{rk}
If~$k$ is a~$p$-adic field and if we follow the conventions set in remark~\ref{rk:padic_conventions}, the linear algebra for the extraction of a basis of the span of~$S$ at line \ref{alg:divadd:dep} can be performed mod~$p$ to save time.
\end{rk}

\bigskip

We now show how to subtract divisors at the level of the Riemann-Roch spaces. We begin with the case where the dimension of the result is known in advance, which is frequently the case thanks to the Riemann-Roch theorem.

\medskip

\begin{algorithm}[H]
\KwIn{Matrices~$M_A$ and~$M_B$ for the spaces~$\L(A)$ and~$\L(B)$ attached to divisors~$A,B \le 5D_0$ such that~$A-B \le3D_0$ and~$\deg B \ge d_0$, and the dimension~$d$ of~$\L(A-B)$.}
\KwOut{A matrix for the space~$\L(A-B)$.}
$K_A \la$ an equation matrix for~$\L(A)$\;
\Repeat{$\#$ of columns of~$S = d$}{
\For{$j \la 1$ \KwTo~$2$}
{
$b_j \la$ a random combination of columns of~$M_B$\;
$K_j \la K_A$\;
\For{$i \la 1$ \KwTo~$n_Z$}{
Multiply the~$i$-th column of~$K_j$ by the~$i$-th coordinate of~$b_j$\; 
}}
$K \la$ vertical stack of~$K_{V_3}$,~$K_1$, and~$K_2$\;
$S \la$ a kernel matrix for~$K$\; \label{alg:divsub:S}
}
\Return~$S$\;
\caption{DivSub}
\label{alg:divsub}
\end{algorithm}

\begin{proof}
As in algorithm \ref{alg:divadd}, the condition~$A,B,A-B \le 5D_0$ ensure that the elements of the corresponding Riemann-Roch spaces are faithfully represented by their evaluation at~$Z$.

Since~$\deg B \ge d_0 > 2g$, the space~$\L(B)$ is base-point-free; in other words, the inequality
\[ \inf_{f \in \L(B)} (f) \ge -B \]
is actually an equality. It follows that
\[ \L(A-B) = \{ v \in V_5 \ \vert \ v \L(B) \subset \L(A) \}. \]
Actually, we have ~$\L(A-B) = \{ v \in V_3 \ \vert \ v \L(B) \subset \L(A) \}$ since~$A-B \le 3D_0$.

Now,~$b_1$ and~$b_2$ represent random elements~$f_1$ and~$f_2$ of~$\L(B)$, and the matrix~$S$ computed at line \ref{alg:divsub:S} represents~$\ker K_{V_3} \, \cap \, \ker K_1 \, \cap \, \ker K_2$. But the first of these kernels represents precisely~$V_3$ by definition of~$K_{V_3}$, whereas the last two represent~$\{ v \in k^{n_Z}  \ \vert \ v \odot b_j \in \L(A) \}$ for~$j=1,2$. Thus~$S$ represents the subspace~$L = \{ v \in V_3 \ \vert \ v f_1, v f_2 \in \L(A) \}$ of~$V_5$.

If the inequality
\[ \inf\big((f_1),(f_2)\big) \ge -B \]
is actually an equality, then~$L$ is exactly~$\L(A-B)$ by the same logic as above. Else,~$L$ is a strict super-space of~$\L(A-B)$, but we can detect this by comparing its dimension to the known dimension of~$\L(A-B)$, so we simply try again with another choice of~$b_1$ and~$b_2$.
\end{proof}

\begin{rk}
As pointed out in section 3 of~\cite{Mak2}, the condition~$\inf\big((f_1),(f_2)\big) = -B$ can be understood as the condition that two elements of an ideal of a Dedekind domain generate that ideal. This explains why we must take at least two elements~$f_1, f_2 \in \L(B)$. This condition will then be satisfied with high probability, except if~$k$ is a very small finite field, in which case it can be a good idea to consider more than two elements~$f_1, f_2 \in \L(B)$ (cf.~\cite{Bruin} for a detailed analysis). Due to assumption~\eqref{nZ_ratpts},~$k$ cannot be too small, so two elements are enough in practice.
\end{rk}

In the case where~$\dim \L(A-B)$ is not known in advance, we can still compute~$\L(A-B)$ by taking~$b_j$ ranging over the columns~$M_B$, which represent a basis~$(f_j)$ of~$\L(B)$ so that~$\inf (f_j) = -B$. Of course, this leads to solving a larger linear system, and is therefore slower, than with two random elements~$f_1, f_2 \in \L(B)$.

This more difficult case is in particular needed in the following situation: given matrices~$W_D$,~$W_{D'}$ representing points~$x=[D-D_0], x' = [D'-D_0] \in J$, we can test whether~$x = x'$, that is to say whether~$D \sim D'$, since~$\DivSub(W_D,W_{D'})$ represents~$\L(D-D')$ which has dimension 1 in this case, and 0 else. In particular, we can test whether~$W_D$ represents~$0 \in J$ by taking~$W_{D'}=W_0$.

\subsubsection{Arithmetic in the Jacobian}

Thanks to these two basic operations, Makdisi manages to compute in the Jacobian using only linear algebra operations over~$k$. For instance, the following algorithm computes the negative of the sum of two points of~$J$. Since we have the representation~$W_0$ of~$0 \in J$, we can then perform additions and subtractions in~$J$.

\medskip

\begin{algorithm}[H]
\KwIn{Matrices~$W_1$ and~$W_2$ for the spaces~$\L(2D_0-D_1)$ and~$\L(2D_0-D_2)$ encoding the points~$x_1=[D_1-D_0]$ and~$x_2 = [D_2-D_0]$ of~$J$.}
\KwOut{A matrix for a space~$\L(2D_0-D_3)$ encoding~$x_3=[D_3-D_0] \in J$ such that~$x_1 + x_2 + x_3 = 0$.}
$W_{12} \la \DivAdd(W_1,W_2)$ \tcp*{$\L(4D_0-D_1-D_2)$}
$W'_{12} \la \DivSub(W_{12},W_0)$ \tcp*{$\L(3D_0-D_1-D_2)$}
$c \la$ a non-zero combination of columns of~$W'_{12}$ \label{alg:addflip:choose_f} \tcp*{Represents~$f$}
\tcp{where~$(f) = -3D_0+D_1+D_2+D_3$ for some~$D_3 \ge 0$ of degree~$d_0$}
$W_{123} \la c \odot M_V$ \tcp*{$\L(5D_0-D_1-D_2-D_3)$}
$W_3 \la \DivSub(W_{123},W'_{12})$ \tcp*{$\L(2D_0-D_3)$}
\Return~$W_3$\;
\caption{Add-flip}
\label{alg:addflip}
\end{algorithm}

\medskip

Note that on both times that we call DivSub, the dimension of the result, namely~$d_W$, is known in advance by Riemann-Roch.

\bigskip

Finally, although each point of~$J$ is represented by a subspace of~$V$, itself viewed as a subspace of~$k^{n_Z}$, clearly not every~$d_W$-dimensional subspace of~$k^{n_Z}$ corresponds to a point of~$J$. In~\cite{Mak1}, Makdisi gives the following algorithm:

\medskip

\begin{algorithm}[H]
\KwIn{A matrix~$W$ of size~$n_Z \times d_W$ with linearly independent columns representing a subspace of~$V$.}
\KwOut{\True if~$W$ represents a subspace of the form~$\L(2D_0-D)$ for some effective~$D$ of degree~$d_0$, \False else.}
$w \la$ a non-zero combination of columns of~$W$\; \label{alg:membership:w}
$W' \la \DivSub(w \odot M_V,W)$\; \label{alg:membership:critical}
$n \la$ number of columns of $W'$\;
\Return \True if~$n=d_W$, \False if~$n<d_W$\;
\caption{Membership test}
\label{alg:membership}
\end{algorithm}

\begin{proof}
The idea is to check whether the elements of the subspace~$S$ of~$V$ represented by~$W$ have many common zeros. Thus the column vector~$w$ represents a non-zero element~$f$ of this subspace, whose divisor is of the form~$(f)=-2D_0+D+D'$ where~$D$ is the largest divisor such that~$W \subset \L(2D_0-D)$ and~$D'$ is another effective divisor, so that~$W'$ represents~$\L(2D_0-D')$. The larger the locus~$D$ of common zeros of~$S$, the smaller~$D'$, so the larger~$\dim W'$. See~\cite[theorem/algorithm 3.14]{Mak1} for more details.
\end{proof}

Loosely speaking, the goal of sections \ref{sect:diff} and \ref{sect:lift} will be to determine the differential of the map~$W \mapsto W'$, so as to be able to deform the representation of a~$p$-adic point of~$J$ by a matrix with entries mod~$p^e$ into a representation by a matrix mod~$p^{2e}$.

\begin{rk}
Since Makdisi's algorithms involve linear algebra in size~$O(g) \times O(g)$, their complexities are~$O(g^\omega)$ operations in the ground field, where~$2 \le \omega < 3$ is such that two matrices of size~$n \times n$ can be multiplied in~$O(n^\omega)$ operations.
\end{rk}

\subsection{Extra operations}

We now present complements of ours to Makdisi's algorithms, some of which are inspired by~\cite{Bruin}.

\subsubsection{Fast exponentiation and add-flip chains}

Algorithm \ref{alg:addflip} does not compute the sum of two points of~$J$, but rather their ``add-flip'', that is to say the negative of that sum. Therefore the usual notion of \emph{addition chain} for fast exponentiation needs to be adapted.

\begin{de}
Let~$m \in \Z$. An \emph{add-flip chain} for~$m$ of \emph{length}~$n \in \N$ is a finite sequence of triples of integers
\[ (m_0,i_0,j_0), (m_1,i_1,j_1), \cdots , (m_n, i_n, j_n) \]
such that~$m_0 = 0$,~$m_1 = 1$,~$m_n = m$, and for all~$2 \leqslant s \leqslant n$, we have~$0 \leqslant i_s, j_s < s$ and~$m_s = - m_{i_s} - m_{j_s}$.
\end{de}

The value of~$i_0, j_0, i_1, j_1$ does not matter and is left undefined.

\bigskip

Since we have the representation~$W_0$ of~$0 \in J$, we can perform any addition by performing two consecutive add-flips. Therefore, for all~$m \in \Z$ distinct form 0 and 1, their exists an add-flip chain of length~$O(\log \vert m \vert)$; and conversely an add-flip chain for~$m$ must clearly have length at least~$O(\log \vert m \vert)$. 

Short add-flip chains for a given~$m \in \Z$ can be found by putting~$m$ in \emph{non-adjacent form} in the sense of~\cite[section 9.1.4]{hyperell_crypto}. This method has the advantage of being extremely fast. Adapting the continued fraction method described in~\cite{cfracchain} results in a much slower algorithm that yields chains which are on average only a few percent shorter.

\subsubsection{Pairings}

Later in section \ref{sect:galrep}, we will generate points of an~$\Fl$-subspace~$T_\rho \subset J[\ell]$ at random, and we will need a tool to detect linear relations between them so as to know whether we have obtained a generating set of~$T_\rho$ yet. For this, we will use parings.

Let~$m \in \N$. The Weil pairing is defined as
\[ \begin{array}{rccl} e_m:& \bigwedge^2 J(k)[m] & \longrightarrow & \mu_m(k) \\ & (x,y) & \longmapsto & \displaystyle \frac{f_x(D_y)}{f_y(D_x)}, \end{array} \]
where~$D_x \sim x$ and~$D_y \sim y$ are divisors whose supports do not intersect, and where~$f_x$ and~$f_y$ are functions of~$C$ such that~$(f_x) = m D_x$ and~$(f_y)=m D_y$. It has the disadvantage of being skew-symmetric, which means it is useless for our purpose when the space~$T_\rho$ is totally isotropic for example. Therefore, we prefer to use the less-known Frey-R\"uck pairing~\cite{FreyRuck}
\[ \begin{array}{rccl} \{\cdot,\cdot\}_m: & J(k)[m] \times J(k)/m J(k) & \longrightarrow & k^\times / k^{\times m} \\ & (x,y) & \longmapsto & \displaystyle f_x(D_y). \end{array} \]
In the special case where~$k = \F_q$ is a finite field containing~$\mu_m$, we can linearize it by composing it with
\[ \begin{array}{ccccl} k^\times / k^{\times m} & \longrightarrow & \mu_m & \longrightarrow & \Z/m\Z \\ z & \longmapsto & z^{(q-1)/m} & \longmapsto & \log_{\zeta} z^{(q-1)/m} \end{array} \]
where~$\zeta \in \mu_m$ is a primitive~$m$-th root of 1 and~$\log_\zeta$ denotes the corresponding discrete logarithm. We will denote the corresponding version of the pairing by 
\[ [ \cdot , \cdot ]_m : J(\F_q)[m] \times J(\F_q)/mJ(\F_q) \longrightarrow \Z/m\Z, \]
the choice of~$\zeta$ being implicit. This pairing is perfect, and allows us to construct linear forms on~$J[m]$ simply as~$[\cdot,y]_m$ for any~$y \in J(\F_q)$, and thus to detect linear relations satisfied by a finite collection of points of~$J[m]$ if we can generate sufficiently many elements~$y \in J(\F_q)/mJ(\F_q)$.

\bigskip

An implementation of the Frey-R\"uck pairing in Makdisi's framework is described in~\cite{Bruin}. We describe here our own implementation, which is simpler (for instance it stays in~$V_5=\L(5 D_0)$ whereas~\cite{Bruin} goes up to~$\L(7 D_0)$).

We start with the following auxiliary procedure, which computes the value of a function at a divisor, and may thus be viewed as a generalization in Makdisi's framework of the concept of resultant:

\medskip

\begin{algorithm}[H]
\KwIn{A column~$c$ representing a non-zero function~$f \in V_3$ and a matrix~$W_D$ representing~$\L(2D_0-D)$, with~$\deg D = d_0$.}
\KwOut{$f(D)$ if the supports of $D$ and $D_0$ do not intersect, else \FAIL.}
$w_1,\cdots,w_{d_W} \la$ the columns of~$W_D$, where~$d_W=\dim V-d_0$ as before\;
\tcp{Find supplement of $\L(2D_0-D)$ in $V=\L(2D_0)$}
Choose columns~$s_1, \cdots, s_{d_0}$ of~$M_V$ such that~$s_1, \cdots, s_{d_0}, w_1,\cdots,w_{d_W}$ are linearly independent\;
$d'_W \la 4d_0+1-g$ \tcp*{dimension of~$\L(5D_0-D)$}
$w'_1,\cdots,w'_{d'_W} \la$ columns of $\DivAdd(W_D,M_{V_3})$ \tcp*{$\L(5D_0-D)$}
\tcp{Find supplement of $\L(5D_0-D)$ in $V_5 = \L(5D_0)$}
Choose columns~$v_1, \cdots, v_{d_0}$ of~$M_V$ and~$u_1, \cdots, u_{d_0}$ of~$M_{V_3}$ such that $t_1, \cdots, t_{d_0}, w'_1,\cdots,w'_{d'_W}$ are linearly independent, where $t_i \la u_i \odot v_i$ \linebreak for all~$i$\;
\tcp{Find relations}
$A \la$ matrix with columns~$s_1, \cdots, s_{d_0}, t_1, \cdots, t_{d_0}, w'_1,\cdots,w'_{d'_W}$\;
$K \la$ matrix whose columns form a basis of~$\ker A$\; \label{alg:norm:K1}
$K_s \la$ rows~$1$ to~$d_0$ of~$K$\;
$K_t \la$ rows~$d_0+1$ to~$2d_0$ of~$K$\;
$M_1 \la K_t K_s^{-1}$\;
$\Delta_1 \la \det M_1$\;
\lIf{$\Delta_1=0$}{\FAIL}
\tcp{Find relations}
$A \la$ matrix with columns~$c\odot s_1, \cdots, c \odot s_{d_0}, t_1, \cdots, t_{d_0}, w'_1,\cdots,w'_{d'_W}$\;
$K \la$ matrix whose columns form a basis of~$\ker A$\;
$K_s \la$ rows~$1$ to~$d_0$ of~$K$\;
$K_t \la$ rows~$d_0+1$ to~$2d_0$ of~$K$\;
$M_f \la K_t K_s^{-1}$\;
$\Delta_f \la \det M_f$\;
\Return~$\Delta_f/\Delta_1$\;
\caption{Evaluating a function at a divisor}
\label{alg:norm}
\end{algorithm}

\begin{proof}
View $V_5 = \L(5D_0)$ and all its subspaces as embedded into $k^{n_Z}$ as explained in section \ref{sect:Mak}, and define
\[ L_D = \L(2D_0-D), \quad L'_D = \L(5D_0-D) \]
for brevity.

We have~$V = \L(2D_0) = S \oplus L_D$, where~$S$ is the subspace spanned by the~$s_i$; similarly~$V_5 = \L(5D_0) = T \oplus L'_D$ where~$T$ is the subspace spanned by the~$t_i = u_i \odot v_i$.

By definition,
\[ L_D = \{ v \in V \, \vert \, v_{\vert D} = 0 \}, \]
so that~$\O_D \simeq V / L_D \simeq S$; and similarly~$\O_D \simeq V_5 / L'_D \simeq T$.

Since~$f \in V_3 = \L(3D_0)$, multiplication by~$f$ induces a map 
\[ \xymatrix{\mu_f : S \ \ar@{^(->}[r] & V \ar[r]^f & V_5 \ar@{->>}[r] & T} \]
and~$f(D)$ is the determinant of this map seen through the isomorphisms~$S \simeq \O_D$ and~$T \simeq \O_D$.

Unfortunately these isomorphisms are not explicit, so it would make no sense to simply take the determinant of the matrix expressing this map on the~$s_i$ and the~$t_i$. This is why we divide this determinant by that of the ``identity'' map
\[ \xymatrix{\charf: S \ \ar@{^(->}[r] & V \ar[r]^1 & V_5 \ar@{->>}[r] & T} \]
induced by the inclusion of~$V$ into~$V_5$.

More specifically, at line \ref{alg:norm:K1} we find relations between the~$s_i$ and the basis \linebreak~$t_1, \cdots, t_{d_0}, w'_1,\cdots,w'_{d'_W}$ of~$V_5 = T \oplus W'_D$, and we project these relations in~$T = V_5 / W'_{D}$ at the next two lines by dropping the rows of $K$ corresponding to the~$w'_i$. The matrix~$K_s$ formed of the~$s$-rows of~$K$ is necessarily invertible since~$t_1, \cdots, t_{d_0}, w'_1,\cdots,w'_{d'_W}$ are linearly independent, so~$M_1$ is well-defined and is the negative of the matrix of the map~$\charf$ with respect to the bases~$s_i$ and~$t_i$. Its determinant~$\Delta_1$ is zero iff. there exists a nonzero element of~$S$ that lies in~$L'_D$, which means that the supports of~$D$ and~$D_0$ are not disjoint, in which case we return~$\FAIL$. Else, by repeating the previous steps with the~$c \odot s_i$ instead of the~$s_i$, we get $M_f$, the negative of the matrix of~$\mu_f$, and we recover~$f(D)$ as its determinant~$\Delta_f$ divided by~$\Delta_1$.

The complexity of this algorithm is~$O(g^\omega)$.
\end{proof}

\begin{rk}
It is actually more efficient to compute~$\Delta_1$ as~$\det(K_t)/\det(K_s)$ than as~$\det(K_t K_s^{-1})$; we show the less efficient way here only for clarity. The same goes of course for~$\Delta_f$.
\end{rk}

Algorithm \ref{alg:norm} is only valid for~$f\in V_3$, which is in general not the case for the function~$f_x$ in the definition of the Frey-R\"uck pairing. Therefore the implementation of the pairing is still not completely straightforward. Our solution, inspired by~\cite{Bruin}, is presented below; we assume that algorithm \ref{alg:addflip} has been modified so as to return the column~$c$ chosen at line \ref{alg:addflip:choose_f} as well as~$W_3$.

\medskip

\begin{algorithm}[H]
\KwIn{Matrices~$W_T$ and~$W_D$ representing~$\L(2D_0-T)$ and~$\L(2D_0-D)$ respectively, and~$m \in \N$ such that~$[T-D_0] \in J[m]$.}
\KwOut{An element of~$k^\times$ representing~$\{ [T-D_0],[D-D_0]\}_m \in k^\times / k^{\times m}$.}
$(m_0,i_0,j_0), \cdots, (m_n, i_n, j_n) \la$ an add-flip chain for~$m$\;
Initialize two vectors~$H$ and~$T$ indexed by~$0 \cdots n$ with~$H[0] \la 1$,~$H[1] \la 1$,~$T[0] \la W_0$,~$W[1] \la W_T$\;
$f,W_{D'_0} \la$ AddFlip($W_0,W_0$)\;
\For{$s \la 2$ \KwTo~$n$}{
$c_s,W_s \la$ Add-flip($T[i_s],T[j_s]$)\; \label{alg:FR:addflip}
$T[s] \la W_s$\;
$H[s] \la \frac{f_s(D'_0)}{f_s(D) H[i_s] H[j_s]}$, where $f_s$ is the function represented by $c_s$\label{alg:FR:H}\;
\tcp{Using algorithm \ref{alg:norm} twice, on $(c_s,W_D)$ and $(c_s,W_{D'_0})$}
}
$c_\infty \la$ a non-zero combination of columns of~$V$ such that each of the columns of $c_\infty \odot W_0$ lies in the span of the columns of $T[n]$\; \label{alg:FR:certiftors} 
\Return~$H[n] \frac{f_\infty(D)}{f_\infty(D'_0)}$, where $f_\infty$ is the function represented by $c_\infty$\;
\tcp{Using algorithm \ref{alg:norm} twice again}
\caption{Frey-R\"uck pairing}
\label{alg:FR}
\end{algorithm}

\pagebreak

\begin{proof} (Cf.~\cite[p. 39]{Bruin} and~\cite[section 4]{Miller})
For each~$s \le n$, denote by~$T_s$ the unique effective divisor  of degree~$d_0$ such that the matrix~$W_s$ represents the space~$\L(2D_0-T_s)$. Clearly~$T[s] = W_{T_s}$ and~$[T_s-D_0]=m_s [T-D_0]$ for all~$s$. In particular, for~$s=n$ we find that~$[T_n-D_0]=0$, so the space \[ \L(D_0-T_n) = \{ v \in V \ \vert \ v \L(D_0) \subset \L(2D_0-T_n) \} \] is 1-dimensional. As the matrices $W_0$ and $T[n] = W_n$ represent respectively the spaces $\L(D_0)$ and $\L(2D_0-T[n])$, the column~$c_\infty$ computed at line~\ref{alg:FR:certiftors} represents a non-zero element~$f_\infty$ of this space, which therefore satisfies $(f_\infty) = T_n-D_0$ and is unique up to multiplication by a nonzero constant.

Define inductively a sequence of functions~$h_0, \cdots, h_n \in k(C)^\times$ by~$h_0=h_1=1$ and~$h_s=1/h_{i_s} h_{j_s} f_s$ for~$s \ge 2$. Line \ref{alg:FR:H} ensures that~$H[s] = h_s(D-D'_0)$ for all~$s$. Besides, we have~$(f_s) = -3D_0+T_{i_s}+T_{j_s}+T_s$ for all $s$ by construction, so an induction on $s$ shows that
\[ (h_s) = m_s T - T_s - (m_s-1) D_0 \]
for all~$s$. Since~$(f_\infty)=T_n-D_0$, we get~$(h_n / f_\infty) = m (T-D_0)$, so that
\[ \{ [T-D_0],[D-D_0]\}_m = (h_n / f_\infty)(D-D'_0) = H[n]  \frac{f_\infty(D'_0)}{f_\infty(D)} \]
as~$D'_0 \sim D_0$.

The complexity of this algorithm is~$O(g^\omega \log m)$ operations in~$k$, to which one should add the cost of the discrete logarithm in~$\mu_m$ if one wants to use the linearized version of the pairing.
\end{proof}

\begin{rk}
We have introduced the divisor~$D'_0$ because the functions~$f_s$ and~$f_\infty$, being elements of~$\L(3D_0)$, have poles at~$D_0$, and therefore cannot be evaluated there.
\end{rk}

\begin{rk}
This algorithm will fail if algorithm \ref{alg:norm} fails or returns 0. In this case, one should try again, making sure that at line \ref{alg:addflip:choose_f} of algorithm \ref{alg:addflip}, the function~$f$ (which is called~$f_s$ in algorithm \ref{alg:FR}) is chosen randomly, and possibly after replacing~$D$ with a linearly equivalent divisor (using the same strategy as for the construction of~$D'_0$). In practice, failure is rare since we work over a finite field large enough that the curve has enough rational points to satisfy \eqref{nZ_ratpts}.

Note by the way that it is important here to always choose~$f$ randomly. Indeed, the add-flip algorithm \ref{alg:addflip} chooses~$f$ as a nonzero element of a space which is computed as a kernel, and most computer algebra systems give kernels by bases in echelon form. As a result, if for instance we always chose the first basis vector for~$f$, then some coordinates of~$f$ would always be~$0$; in other words~$f$ would always vanish at some fixed points, which would lead to systematic failure of algorithm \ref{alg:norm}.
\end{rk}

\begin{rk}
It is straightforward to modify algorithm \ref{alg:FR} so as to compute Weil pairings if desired.
\end{rk}

\subsubsection{Local charts and evaluation maps}

We now give a construction of Galois-equivariant rational maps~$J \dashrightarrow \A^1$ inspired by the method for modular curves described in~\cite[section 3.6]{algo}. In this construction, we have in mind the case where the ground field is a number field~$k_0$ and where we deal with points of~$J(\overline{k_0})$, and our goal is to obtain a map yielding values of reasonable arithmetic height (cf~\cite{algo} for details). Of course, condition \ref{nZ_ratpts} will not, in general, be satisfiable if we insist that the divisor~$Z$ be formed of points in~$C(k_0)$, so we use points over a~$p$-adic field~$k \supset k_0$ instead. In other words, we actually apply Makdisi's algorithms to the base change of~$C$ from~$k_0$  to~$k$. 

We begin by constructing a map with values in the projective space~$\PP(V)$. For this, we assume that we are able to find effective divisors~$E_1$ and~$E_2$, both of degree~$d_0-g$. Then~$D_0-E_1$ has degree~$g$, so given a sufficiently generic point~$x \in J$, Riemann-Roch shows that there exists a generically unique effective divisor~$E_x$ of degree~$g$ such that
\begin{equation}
[D_0-E_1-E_x] = x.
\label{eqn:eval:Ex}
\end{equation}
Another application of Riemann-Roch then shows that the space~$\L(2D_0-E_1-E_2-E_x)$ has dimension 1, so we get a map
\begin{equation}
\begin{array}{rcl} J & \dashrightarrow & \PP(V) \\ x &  \longmapsto & \L(2D_0-E_1-E_2-E_x) \end{array}
\label{eqn:eval}
\end{equation}
which is only a rational map due to the genericity assumptions on~$x$.

The point of this definition is that in Makdisi's framework, there are many divisors~$D$ representing a given~$x \in J$, that is to say such that~$[D-D_0]=x$. This map is thus useful to turn the representation~$\L(2D_0-D)$ of~$x$ into something that does not depend on~$D$, but only on~$x$. Implementing it is fairly straightforward:

\medskip

\begin{algorithm}[H]
\KwIn{A matrix~$W_D$ for the space~$\L(2D_0-D)$ representing a point~$x=[D-D_0] \in J$, and matrices~$W_1$ and~$W_2$ for the spaces~$\L(2D_0-E_1)$ and~$\L(2D_0-E_2)$, respectively.}
\KwOut{The value in~$\PP(V)$ of the map \eqref{eqn:eval} at~$x$.}
$S_1' \la \DivAdd(W_D,W_1)$ \tcp*{$\L(4D_0-D-E_1)$}
$S_1 \la \DivSub( S'_1, M_V)$ \tcp*{$\L(2D_0-D-E_1)$} 
\lIf{number of columns of $S_1 > 1$}{\Return \FAIL}
$s_1 \la$ the lone column of~$S_1$ \tcp*{$(s_1)=-2D_0+D+E_1+E_x$}
$S'_2 \la \DivSub(s_1 \odot M_V,W_D)$ \tcp*{$\L(2D_0-E_1-E_x)$} \label{alg:eval:Arman}
$S''_2 \la  \DivAdd(S'_2, W_2)$ \tcp*{$\L(4D_0-E_1-E_2-E_x)$}
$S_2 \la  \DivSub(S''_2,M_V)$ \tcp*{$\L(2D_0-E_1-E_2-E_x)$}
\lIf{number of columns of $S_2 > 1$}{\Return \FAIL}
\Return~$S_2$\;
\caption{Evaluation of a rational map~$J \dashrightarrow \PP(V)$}
\label{alg:eval}
\end{algorithm}

\begin{proof}
The matrix $S_1$ represents the space $\L(2D_0-D-E_1)$, which is generically of dimension 1 by Riemann-Roch. Thus~$s_1$ is well-defined up to multiplication by a constant, and its divisor is~$(s_1)=-2D_0+D+E_1+E_x$.
\end{proof}

\begin{rk}\label{rk:rationality_D0_Ei}
We must ensure that the effective divisors~$D_0$,~$E_1$, and~$E_2$ are actually defined over~$k_0$ (as opposed to~$k$) if we want the map~\eqref{eqn:eval} to be~$\Gal(\overline{k_0} / k_0)$-equivariant. For~$D_0$, this is not a problem in practice, since the only requirement that we need to satisfy is that~$d_0 \overset{\text{def}}{=} \deg D_0 \ge 2g+1$ (however we should try to chose~$D_0$ so that~$d_0$ is as close to~$2g+1$ as possible for efficiency). On the other hand,~$E_1$ and~$E_2$ must be of degree exactly~$d_0-g$, and this condition can be more difficult to satisfy over a number field; but if we have to, we can take~$E_1$ and~$E_2$ to be defined over a finite extension of~$k_0$, and form symmetric combinations of the values obtained for these divisors and their Galois conjugates.
\end{rk}

\begin{rk}\label{rk:chart}
The map~\eqref{eqn:eval} may be used as a local coordinate chart on~$J$, and for this purpose,~$E_1$ and~$E_2$ need not be defined over~$k_0$. One should however be aware that the domain on which this map is an immersion may be smaller than its domain of definition.
\end{rk}


We now show how to turn the map \eqref{eqn:eval} into an~$\A^1$-valued map. In~\cite{algo}, we chose two rational points~$P, Q \in C(k)$ away from the supports of~$D_0$,~$E_1$, and~$E_2$, and we composed the map \eqref{eqn:eval} with the map~$s_2 \longmapsto s_2(P)/s_2(Q)$, were~$s_2$ is a nonzero element of the 1-dimensional space~$\L(2D_0-E_1-E_2-E_x)$. This approach was suitable for modular curves, since they have plenty of rational cusps, but is harder to adapt to a general curve; we therefore propose a different approach.

We fix a basis~$v_1, v_2, \cdots$ of~$V$, where the~$v_i$ are defined over~$k_0$ (again for Galois-equivariance reasons). For instance, if~$C$ is given to us by a plane equation~$f(x,y)=0$ over~$k_0$, we will have obtained the~$v_i$ as explicit elements of~$k_0(x,y)$ by a Riemann-Roch space computation over~$k_0$. We then write
$s_2 = \sum \lambda_i v_i$, where the~$\lambda_i$ are scalars, and simply return the ratio~$\lambda_{i_1}/\lambda_{i_2}$, where~$i_1$ and~$i_2$ are fixed indices. This is well-defined, since~$s_2$ is well-defined up to multiplication by a constant, and straightforward to implement in Makdisi's framework. In summary, our~$\A^1$-valued map is
\begin{equation}
\alpha \colon \begin{array}{rcl} J & \dashrightarrow & \A^1 \\ x &  \longmapsto & \lambda_{i_1}/\lambda_{i_2}, \end{array}
\label{eqn:eval:A1}
\end{equation}
where the~$\lambda_i$ are such that~$\sum_i \lambda_i v_i$ spans~$\L(2D_0-E_1-E_2-E_x)$.

The complexity of one map evaluation is~$O(g^\omega)$ operations in~$k$.

\begin{rk}
When we compute Galois representations, we will evaluate~$\alpha$ at the points of an~$\Fl$-subspace~$T_\rho \subset J[\ell]$, and we will need~$\alpha$ to be injective on~$T_\rho$ so that the Galois action on its values faithfully reflects the Galois action on the points of~$T_\rho$. Then the polynomial
\[ F(x) = \prod_{t \in T_\rho} \big(x-\alpha(t)\big) \]
will have coefficients in~$k_0$, and its splitting field over~$k_0$ will be a subfield of the field cut out by the Galois representation, and will agree with this field with high probability. We can ensure that this is indeed the case by checking that~$F(x)$ is squarefree.

We must ensure that the divisors~$E_1$ and~$E_2$ are such that~$[E_1-E_2] \not \in T_\rho$ if~$\ell \neq 2$. Indeed, the divisor of~$s_2$ is of the form~$(s_2)=-2D_0+E_1+E_2+E_x+E'_x$ where~$E'_x$ is effective of degree~$g$, and \eqref{eqn:eval:Ex} shows that
\[ x = [D_0-E_1-E_x] = -[D_0-E_2-E'_x]. \]
Thus if there is a~$t \in T_\rho$ such that~$[E_1-E_2] = t$, then~$E'_x = E_{t-x}$, so that the map~$x \mapsto s_2$ cannot be injective on~$T_\rho$ unless~$\ell=2$.
\end{rk}

\begin{rk}
In particular,~$E_1$ and~$E_2$ should be distinct. In some cases, we may be able to find a effective divisor~$E_1$ of degree~$d_0-g$ defined over~$k_0$, but not a second one; in this case, we can replace the map \eqref{eqn:eval} with the map 
\begin{equation}
\begin{array}{rcl} J & \dashrightarrow & \operatorname{Grass}_{d_W}(V) \\ x &  \longmapsto & \L(2D_0-E_1-E_x), \end{array}
\label{eqn:eval2}
\end{equation}
where $\operatorname{Grass}_{d_W}(V)$ is the Grassmanian of~$V$ parametrising subspaces of dimension~$d_W = d_0+1-g$; in other words, interrupt algorithm \ref{alg:eval} at line \ref{alg:eval:Arman}. Note that by definition,~$E_1+E_x$ is the only effective divisor~$E \ge E_1$ such that~$[E-D_0] = x$, so that the space~$\L(2D_0-E_1-E_x)$ may be viewed as the normal form representation of~$x$ (with respect ot~$E_1$) in Makdisi's framework.

The advantage of this approach is that we need only one~$k_0$-rational effective divisor of degree~$d_0-g$, and that it requires half less computation; its disadvantage is that \eqref{eqn:eval2} now takes values in the Grassmanian of~$V$, since~$\dim \L(2D_0-E_1-E_x) = d_W > 1$ for generic~$x$. We can still use Grassmanian coordinates with respect to the basis~$(v_i)$ of~$V$ to turn \eqref{eqn:eval2} into an~$\A^1$-valued map, but experiment shows that the values thus obtained have a worse arithmetic height. In what follows, we will only use \eqref{eqn:eval:A1}.
\end{rk}

\subsubsection{Makdisi over~$p$-adic fields}\label{sect:Howell}

We will later apply Makdisi's algorithms over the ground field~$k=\Q_q$. This means that we must be able to perform linear algebra over~$k$, namely kernel and equation matrices (these are actually the same up to transposition), which is a non-trivial problem since we can only represent elements of~$k$ on a computer with finite~$p$-adic accuracy.

For this, we use Page's recent implementation in~\cite{gp} of the computation of kernels in Howell form over rings of the form~$R=\Z_q/p^e$, where~$\Z_q$ is the ring of integers of~$\Q_q$ and~$e \in \N$. Recall that the Howell form of an~$R$-submodule~$M$ of~$R^d$~($d \in \N$) is a canonical generating set of this module, which is essentially characterized by the fact that these generators are in reduced echelon form, and that for all $d' \le d$, each element of~$M$ whose first~$d'$ entries are~$0$ is an $R$-combination of the generators whose first~$d'$ entries are~$0$; cf.~\cite{Howell} or~\cite{Storjohann} (beware that these references consider row spans, whereas we deal with column spans in this article). The following result shows that if we consider an approximation up to~$O(p^e)$ of a matrix~$A$ with coefficients in~$\Z_q$ which has ``good reduction'', i.e. whose rank does not decrease after reduction mod~$p$, then we can compute a kernel matrix of~$A$ with the same accuracy and satisfying the same conditions, and thus use Makdisi's algorithms without loss of~$p$-adic accuracy.

\begin{thm}\label{thm:Howell}
Let~$k$ be a non-Archimedian local field with ring of integers~$\O$, uniformizer~$\varpi$ and residue field~$\kappa$, and let~$A$ be an~$m \times n$ matrix with coefficients in~$\O$, representing a linear map~$T : \O^n \rightarrow \O^m$. Let~$e \in \N$, and let~$K_e$ be the kernel of the induced map~$T_e : (\O/\varpi^e)^n \rightarrow (\O/\varpi^e)^m$. If~$\rank_k(A) = \rank_\kappa (A \bmod \varpi)$, then the Howell generators of~$K_e$ that are nonzero mod~$\varpi$ form an approximation mod~$\varpi^e$ of an~$\O$-basis of~$\ker T$.
\end{thm}

\begin{rk}
There can be Howell generators of~$K_e$ that are~$0 \bmod \varpi$, as demonstrated by the example~$A = (1 \ \ \varpi)$,~$e=2$.
\end{rk}

\begin{proof}
Let us write~$K_\infty = \ker T$, and~$K_k$ for the kernel of the induced map~$k^n \rightarrow k^m$, so that~$K_\infty = K_k \, \cap \, \O^n$. Since~$\O$ is a PID, we can consider an~$\O$-basis~$h_1, \cdots, h_d \in \O^n$ of~$K_\infty$ in Hermite normal form.

We claim that the~$h_i$ are linearly independent mod~$\varpi$. Indeed, let~$\lambda_i \in \O$, and let~$x = \sum_i \lambda_i h_i \in K_\infty$. If~$x \equiv 0 \bmod \varpi$, say~$x = \varpi y$ for some~$y \in \O^n$, then~$y \in \O^n \cap K_k = K_\infty$, which shows that the~$\lambda_i$ are all 0 mod~$\varpi$ since the~$h_i$ are linearly independent over~$k$.

We can then prove by induction on~$e$ that the~$h_i$ generate~$K_e$. Indeed, for~$e=1$ this follows from the above and the fact that the rank of~$A$ does not decrease after reduction mod~$\varpi$. Suppose this is true for~$e$, and let~$x  \in  \O^n$ be such that~$x \bmod \varpi^{e+1}$ lies in~$K_{e+1}$. Then~$x \bmod \varpi^e$ lies in~$K_e$, so there exist~$\lambda_i \in \O$ and~$y \in \O^n$ such that~$x = \sum_i \lambda_i h_i + \varpi^e y$; therefore
\[ 0 \equiv Ax = \varpi^e A y \bmod \varpi^{e+1}. \]
Dividing by~$\varpi^e$ shows that~$y \bmod \varpi$ lies in~$K_1$, and is thus in the~$\O$-span of the~$h_i$, and so is~$x$.

It follows that the Howell form of~$K_e$ agrees with that of the~$h_i \bmod \varpi^e$; however the~$h_i$ are in echelon form, so they are part of their Howell form mod~$\varpi^e$, and the other vectors must be linear combinations of them that are 0 mod~$\varpi$. 
\end{proof}

\begin{rk}
It is shown in~\cite{Storjohann} that the Howell form of~$\ker A \bmod p^e$ can be computed with the same complexity as if~$R=\Z_q/p^e$ were a field, namely~$O(n^{\omega})$ operations in~$R$ for~$A$ of size~$n \times n$. Therefore, we can execute the algorithms presented in this section over~$R$ with the same number of operations in~$R$ as if~$R$ were a field.
\end{rk}

\subsubsection{Galois action}\label{sect:Frob}

We now suppose that the curve~$C$ is defined over~$\Q_p$ (in practice, it is actually defined over~$\Q$), and that we have a matrix~$W_D$ with coefficients in~$\Z_q/p^e$ representing a point~$x \in J(\Z_q/p^e)$ for some~$e \in \N$ (in particular, if~$e=1$ then~$W_D$ represents a point in~$J(\F_q)$). We would like to compute a matrix~$W_{D'}$ representing~$x^{\Phi}$ to the same $p$-adic accuracy, where~$\Phi \in \Gal(\Q_q/\Q_p)$ is the Frobenius. This is easy, provided that the divisors~$D_0$ and~$Z$ are defined over~$\Q_p$, and that we know how~$\Phi$ acts on the support of~$Z$ (in practice this information should be recorded at the creation of~$Z$).

\medskip

\begin{algorithm}[H]
\KwIn{A matrix~$W_D$ representing a point~$x \in J(\Q_q)$.}
\KwOut{A matrix~$W_{D'}$ representing the point~$x^\Phi \in J(\Q_q)$.}
$W \la W_D^\Phi$\;
Permute the rows of~$W$ by the permutation induced by~$\Phi^{-1}$ on~$Z$\;
\Return~$W$\;
\caption{Frobenius}
\label{alg:Frob}
\end{algorithm}

\begin{proof}
Recall that the~$j$-th column of~$W_D$ represents a function~$w_j$ such that the~$w_j$ form~$\Q_q$-basis of the space~$\L(2D_0-D)$ representing~$x$, in that the~$i$-th row of~$W_D$ contains the value of the~$w_j$ at the~$i$-th point in the support of~$Z$. Since~$x=[D-D_0]$, we have~$x^\Phi = [D^\Phi-D_0]$, so~$x^\Phi$ is represented by~$\L(2D_0-D^\Phi) = \L(2D_0-D)^\Phi$, but~$v^\Phi(P) = \big(v(P^{\Phi^{-1}})\big)^\Phi$ for all~$P \in C(\Q_q)$.

The complexity of this algorithm is of course~$O(g^2)$ operations in~$\Q_q$.
\end{proof}

\begin{rk}
We use a model for~$\Q_q$ of the form~$\Q_p(\theta)/T(\theta)$, where~$T(x) \in \Z[x]$ is a lift of an irreducible polynomial over~$\F_p$. In this model, every element of~$\Q_q$ is of the form~$\sum_i c_i \theta^i$ with~$c_i \in \Q_p$, and the image of such an element by~$\Phi$ is simply~$\sum_i c_i \theta'^i$, where~$\theta'$ is the unique root of~$T(x)$ congruent to~$\theta^p$ mod~$p$.
\end{rk}

\section{Differentiation of linear algebra operations}\label{sect:diff}

Let us now suppose that the ground field is a finite extension~$k$ of~$\Q_p$ with ring of integers~$\O$ and uniformizer~$\varpi$ (in practice,~$k=\Q_q$,~$\O=\Z_q$, and~$\varpi=p$). Makdisi's algorithms rely on linear algebra; in particular, they frequently involve the computation of a kernel matrix or of an equation matrix (a.k.a. left kernel) of a given matrix~$A$ with coefficients in~$\O$. The goal of this section is to determine how these deform under~$\varpi$-adically small perturbations of~$A$.

However, the problem thus presented is ill-defined, since a given vector space (in this case, the right or left kernel of~$A$) admits many different bases in general. We therefore start by rigidifying the notions of kernel matrix and of equation matrix defined at page \pageref{def_eqnmatrix}. For simplicity, we restrict ourselves to the case where the matrix~$A$ has full rank.

Let us treat the case of equation matrices first, since only this case will be used later. We assume that we have set a rule that, given a matrix~$_nA_r$ (this notation means that~$A$ has~$n$ rows and~$r$ columns) with coefficients in~$\O$ and whose reduction mod~$\varpi$ has rank~$r$, assigns to it a \emph{supplement matrix}, that is to say a matrix~$_nS_{n-r}$ with coefficients in~$\O$ such that the augmented matrix~$A^+ = \left( _nA_r \ \vert \ _nS_{n-r} \right)$ is invertible over~$\O$, and which only depends on the reduction of~$A$ mod~$\varpi$. For instance, we can impose that the~$j$-th column of~$S$ has a~$1$ at row~$i_j$ and~$0$ everywhere else, in such a way that~$i_1 < \cdots < i_{n-r}$ and that the~$i_j$ are as small as possible with respect to some total ordering; this is what~\cite{gp}'s function \verb?matsupplement? does.

Let us decompose the inverse~$B=(A^+)^{-1}$ of~$A^+$ over~$\O$ as
\[ B = \left( \begin{array}{c} _rF_n \\ \hline _{n-r}E_n \end{array} \right). \]
We have
\[ I_n = B A^+ =  \left( \begin{array}{c|c} _rFA_r & _rFS_{n-r} \\ \hline _{n-r}EA_r & _{n-r}ES_{n-r}  \end{array} \right), \]
whence the relations~$FA = I_r$,~$FS = {}_r0_{n-r}$,~$EA = {}_{n-r}0_r$, and~$ES = I_{n-r}$. In particular, the rows of~$E$ are linearly independent (both over $k$ and over its residue field), so that~$E$ is an equation matrix for~$A$. We call~$E$ \emph{the} equation matrix of~$A$, and denote it by~$E = \Eqn(A)$, the rule determining~$S$ being implicit. Besides, we call~$F$ the \emph{F-complement} of~$A$.

The advantage of this definition is that it is easy to differentiate~$\Eqn(A)$ with respect to~$A$. Indeed, let~$_nH_r$ be a matrix of the same size as~$A$ and whose entries lie in~$\varpi^e \O$ for some integer~$e \ge 1$. The matrix~$A+H$ is congruent to~$A$ mod~$\varpi$, so the augmented matrix~$ \left( A+H \ \vert \ S \right)$ is still invertible (with the same~$S$) over~$\O$, whence~$(A+H)^+ = A^++H^0$ where~$H^0 = \left( _nH_r \ \vert \ _n0_{n-r} \right)$. Therefore
\begin{align*}
\big((A+H)^+\big)^{-1} &= B - B H^0 B + O(\varpi^{2e}) \\
&= \left( \begin{array}{c} _rF_n \\ \hline _{n-r}E_n \end{array} \right) - \left( \begin{array}{c} _rF_n \\ \hline _{n-r}E_n \end{array} \right) \left( _nH_r \ \vert \ _n0_{n-r} \right) \left( \begin{array}{c} _rF_n \\ \hline _{n-r}E_n \end{array} \right) + O(\varpi^{2e}) \\
&= \left( \begin{array}{c} _rF_n \\ \hline _{n-r}E_n \end{array} \right) - \left( \begin{array}{c} _rFHF_n \\ \hline _{n-r}EHF_n \end{array} \right) + O(\varpi^{2e}), \\
\end{align*}
which shows that
\begin{equation}
\Eqn(A+H) = \Eqn(A) - \Eqn(A) H F + O(\varpi^{2e}) \label{eqn:diffeqn}
\end{equation}
where~$F$ is the~F-complement of~$A$ and~$O(\varpi^{2e})$ denotes a matrix whose entries lie in~$\varpi^{2e} \O$.

The case of kernels, being the transpose of the previous case, is similar: given~$_rA_n$ with coefficients in $\O$ and whose reduction mod $\varpi$ has rank~$r$, we set
\[ A^+ = \left( \begin{array}{c} _rA_n \\ \hline _{n-r}S_n \end{array} \right) \]
where~$S$ is determined by similar rules so that~$A^+$ is invertible, and then if
\[ (A^+)^{-1} = \left( _nL_r \ \vert \ _{n}K_{n-r} \right) \]
then~$K$ is a kernel matrix of~$A$, which we 
denote by~$K = \Ker(A)$ with a capital K; besides we have
\begin{equation}
\Ker(A+H) = \Ker(A) - LH \Ker(A) + O(\varpi^{2e}) \label{eqn:diffker}
\end{equation}
for~$H$ of the same size as~$A$ and with coefficients in~$\varpi^e \O$.

\begin{rk}
In the next section, we will work with matrices with coefficients in~$R = \Z_q/p^{2e}$, and our perturbation matrices~$H$ will have coefficients in~$I = p^{e} R$. Since~$I^2 = 0$, all the~$O(\varpi^{2e})$ terms will vanish in~$R$.
\end{rk}


\section{$p$-adic lifting of points of~$J$}\label{sect:lift}

In this section, we suppose that we are given a matrix~$W$ with coefficients in~$\Z_q/p^e$ representing a point of~$J(\Z_q/p^e)$ (in other words, a point of~$J(\Q_q)=J(\Z_q)$ with absolute accuracy~$O(p^e)$) for some~$e \in \N$, whereas the data representing~$J$ itself have been precomputed with accuracy at least~$O(p^{2e})$ (meaning that the matrices~$K_{V_3}$,~$M_V$, etc., have coefficients in~$\Z_q/p^{e'}$ where~$e' \ge 2e$). Our goal is to find a deformation of~$W$ over~$\Z_q/p^{2e}$ which represents a point of~$J(\Z_q/p^{2e})$. Once we have solved this problem, we will be able to lift quadratically a representation of a point in~$\bar x \in J(\F_q)$ (given by a matrix for~$W_D$ with coefficients in~$\F_q$) to a representation to arbitrary~$p$-adic accuracy of a point of~$\widetilde x \in J(\Q_q)$ which reduces to~$\bar x$ mod~$p$.

Our approach consists in determining how a deformation of the input of the membership test (algorithm \ref{alg:membership}) perturbs the execution of this algorithm. Let us introduce some notation: given a commutative ring~$R$ and a column vector~$c \in  R^{n_Z}$, we will denote by~$\Delta_c$ the diagonal matrix of size~$n_Z$ and coefficients in~$R$ whose diagonal is formed of the coefficients of~$c$. Also define~$r=n_Z-d_W$.

Let~$S$ be a matrix of size~$n_Z \times d_W$ and coefficients in~$\Z_q$, whose columns~$s_i \in \Z_q^{n_Z}$ are independent mod~$p$ and form a basis of a~$\Q_q$-subspace of~$V$ of dimension~$d_W$ (this is the case in particular if~$S$ is an arbitrary lift of~$W$ still representing a subspace of~$V$, in view of the conventions established in remark~\ref{rk:padic_conventions}). An analysis of algorithm~\ref{alg:membership} shows that if we called this algorithm on~$S$, then the matrix~$W'$ introduced at line~\ref{alg:membership:critical} would be computed as the kernel of the matrix
\[ \left( \begin{array}{c} K_{V_3} \\ \hline \vdots \\ \Eqn(s_1 \odot M_V) \Delta_{s_{m}} \\ \vdots \end{array} \right) \]
where~$m$ ranges from~$1$ to~$d_W$ and where we assume without loss of generality that the~$w$ chosen at line \ref{alg:membership:w} is~$s_1$. Indeed, the block of equations~$K_{V_3}$ describes the column vectors~$c \in \Q_q^{n_Z}$ representing elements of~$V_3 = \L(3 D_0)$, whereas for each~$m$ the block of equations~$\Eqn(s_1 \odot M_V) \Delta_{s_m}$ describes the~$c$ such that~$s_m \odot c$ represents an element of~$f_1 V = f_1 \L(2D_0)$, where~$f_1$ is the function represented by~$s_1$. In particular, for~$m=1$ this block is redundant with~$K_{V_3}$ since~$D_0 \ge 0$; we therefore define
\[ K(S) = \left( \begin{array}{c} K_V \\ \hline \vdots \\ \Eqn(s_1 \odot M_V) \Delta_{s_{m}} \\ \vdots \end{array} \right) \]
where~$m$ ranges from~$2$ to~$d_W$. Algorithm \ref{alg:membership} shows that~$S$ represents an element of~$J(\Q_q)$ if and only if the~$\Q_q$-rank of~$K(S)$ is~$r$ and not more.

Since by assumption the matrix~$W$ represents a point of~$J(\Z_q/p^e)$, the matrix~$K(W)$ is thus the reduction mod~$p^e$ of a matrix with coefficients in~$\Z_q$ and~$\Q_q$-rank~$r$.

Let now~$\widetilde W$ be a arbitrary lift of~$W$ to~$\Z_q/p^{2e}$ which still represents a subspace of~$V$. For instance, such a lift can be obtained by finding a matrix~$U$ with coefficients in~$\Z_q/p^e$ such that~$W = M_V U$, and taking~$\widetilde W = M_V \widetilde U$ where~$\widetilde U$ is an arbitrary lift of~$U$ to~$\Z_q/p^{2e}$. The matrix~$K(\widetilde W)$ is then a lift of~$K(W)$ to~$\Z_q/p^{2e}$, and~$\widetilde W$ represents a point~$\widetilde x \in J(\Z_q/p^{2e})$ if and only if~$K(\widetilde W)$ is the reduction mod~$p^{2e}$ of a matrix with coefficients in~$\Z_q$ and of $\Q_q$-rank~$r$ (and~$\widetilde x$ is then necessarily a lift of the point~$x \in J(\Z_q/p^e)$ represented by~$W$).

Enforcing this rank condition in terms of vanishing of minors would be inefficient; instead we use the following method.

\begin{lem}\label{lem:mat_rank}
Let~$_mM_n = \left( \begin{array}{c|c} _rA_r & _rB_{n-r} \\ \hline _{m-r}C_r & _{m-r}D_{n-r} \end{array} \right)$ be an~$m \times n$ matrix with coefficients in a field such that the~$r \times r$ block~$A$ is invertible. Then~$\rank M \ge r$, with equality iff.
\[ D-CA^{-1}B = {}_{m-r}0_{n-r}. \]
\end{lem}

\begin{proof}
The matrix~$_nP_n = \left( \begin{array}{c|c} I_r & -A^{-1}B \\ \hline 0 & I_{n-r} \end{array} \right)$ is clearly invertible, and
\[ MP = \left( \begin{array}{c|c} A & 0 \\ \hline C & D-CA^{-1}B \end{array} \right). \]
\end{proof}

Let~$\overline K$ be the reduction of~$K(\widetilde W)$ mod~$p$. It has coefficients in~$\F_q$, and does not depend on the lift~$\widetilde W$ of~$W$. Since~$W$ represents a point of~$J(\Z_q/p^e)$ with~$e \ge 1$, the $\F_q$-rank of~$\overline K$ is exactly~$r$; indeed~$\overline K = K(\overline W)$, where~$\overline W$ is the reduction of~$W$ mod~$p$, which does represents a point of~$J(\F_q)$. Therefore,~$\overline K$ has an invertible minor of size~$r$. Possibly after permuting the rows and columns of~$K(\widetilde W)$, we may assume that the top-left~$r \times r$ block of~$\overline K$ is invertible over~$\F_q$; we will suppose here that this is directly the case for the clarity of exposition. Then the top left~$r \times r$ block of~$K(\widetilde W)$ is invertible over~$\Z_q/p^{2e}$. Splitting
\[ K(\widetilde W) = \left( \begin{array}{c|c} A & B \\ \hline C & D \end{array} \right) \]
as in lemma \ref{lem:mat_rank}, we have therefore established that
\begin{equation}
\widetilde W \text{ represents a point of } J(\Z_q/p^{2e}) \Longleftrightarrow D-CA^{-1}B = 0 \bmod p^{2e}. \label{eqn:lift_criterion}
\end{equation}

Furthermore, we can determine how~$K(\widetilde W)$ deforms with~$\widetilde W$, thanks to the formulas established in section \ref{sect:diff}. Indeed, let~$\widetilde w_1, \cdots, \widetilde w_{d_W}$ be the columns of~$\widetilde W$, so that
\[ K(\widetilde W) = \left( \begin{array}{c} K_V \\ \hline \vdots \\ \Eqn(\widetilde w_1 \odot M_V) \Delta_{\widetilde w_{m}} \\ \vdots \end{array} \right). \]
If we replace~$\widetilde W$ with a deformation~$\widetilde W + M_V H$ still representing a subspace of~$V$, where~$H$ is a matrix with coefficients in~$p^e \Z_q / p^{2e}$ whose columns we denote by~$h_1, \cdots, h_{d_W}$, then for each~$m$, the column~$\widetilde w_m$ gets replaced by~$\widetilde w_m + M_V h_m$, so~$\Delta_{\widetilde w_m}$ gets shifted by~$\Delta_{M_V h_m}$; whereas~$\widetilde w_1 \odot M_V = \Delta_{\widetilde w_1} M_V$ gets shifted by~$\Delta_{M_V h_1}M_V$, so that~$\Eqn(\widetilde w_1 \odot M_V)$ gets shifted by~$-\Eqn(\widetilde w_1 \odot M_V) \Delta_{M_V h_1}M_V F$ by~\eqref{eqn:diffeqn}, where~$F$ is the~F-complement of~$\Eqn(\widetilde w_1 \odot M_V)$ as defined in section~\ref{sect:diff}. Thus in total,~$K(\widetilde W)$ gets shifted by
\begin{equation} \left( \begin{array}{c} 0 \\ \hline \vdots \\ \Eqn(\widetilde w_1 \odot M_V) \Big( \Delta_{M_V h_{m}}- \Delta_{M_V h_1}M_V F \Delta_{\widetilde w_{m}} \Big) \\ \vdots \end{array} \right). \label{eqn:deform_K} \end{equation}

It is therefore easy to see how the blocks~$A$,~$B$,~$C$, and~$D$ deform in  terms of the~$h_m$, and thus to deduce linear equations that the~$h_m$ must satisfy for~$D-CA^{-1}B$ to vanish mod~$p^{2e}$.

\bigskip

These observations translate into the following algorithm:

\medskip

\begin{algorithm}[H]
\KwIn{A matrix~$W$ over~$\Z_q/p^e$ representing a point in~$J(\Z_q/p^e)$.}
\KwOut{A matrix~$\widetilde W \equiv W \bmod p^e$ over~$\Z_q/p^{2e}$ representing a lift of this point to~$J(\Z_q/p^{2e})$.}
\tcp{Lift~$W$ so that it still represents a subspace of~$V$}
$K \la$ kernel of the horizontal concatenation~$(W \vert M_V) \bmod p^{e}$\;
Drop the first~$d_W$ rows of~$K$ and lift the rest to~$\Z_q/p^{2e}$\;
$\widetilde W \la M_V K \bmod p^{2e}$\; \label{alg:lift:Wlift}
$w_1, \cdots, w_{d_W} \la$ the columns of~$\widetilde W$\;
\tcp{Save on number of variables}
$I \la$ a subset of~$\{1,\cdots,n_Z\}$ of size~$d_W$ such that the corresponding rows of~$W$ form an invertible matrix\; \label{alg:lift:I}
$V_0 \la$ a matrix representing the subspace~$\{ v \in V \ \vert \ v(P_i)=0 \ \forall i \in I \}$ of~$V$, where the~$P_i$ are as in section \ref{sect:Mak}\; \label{alg:lift:V0} 
$v_1, \cdots, v_{d_0} \la$ column vectors of~$V_0$\;
\tcp{Evaluate~$D-CA^{-1}B$}
$E,F \la$ equation matrix and~F-complement of~$w_1\odot M_V$\;
$K \la$ stack of~$K_V$ and of the~$E \Delta_{w_m}$ for~$2 \le m \le d_W$\tcp*{$K(\widetilde W)$}
$A,B,C,D \la$ splitting of~$K$ as in lemma~\ref{lem:mat_rank}\;
$R \la (D-CA^{-1}B)/p^e \bmod p^e$\;
\tcp{Precompute some data}
\For{$m \la 2$ \KwTo~$d_W$}{
$F_m \la -M_V F \Delta_{w_m} \bmod p^e$\;
}
\tcp{See the effect of deforming~$w_j$ by~$p^e v_i$ for each~$j \le d_W$,~$i \le d_0$}
\For{$i \la 1$ \KwTo~$d_0$}{
$E_i \la E \Delta_{v_i} \bmod p^e$\;
 \For{$j \la 1$ \KwTo~$d_W$}{
  \eIf{$j=1$}{
  $k \la$ a matrix with the same shape as~$K$, formed by stacking a block of 0 instead of~$K_V$ and the~$E_i F_m$ instead of the~$E \Delta_{w_m}$  for~$2\le m \le d_W$\;
  }
  {
  $k \la$ a matrix with the same shape as~$K$, with~$E_i$ instead of~$E \Delta_{w_j}$ and~$0$ elsewhere\;
  }
 $a,b,c,d \la$ splitting of~$k$\;
 $r_{i,j} \la d-cA^{-1}B+CA^{-1}aA^{-1}B-CA^{-1}b$ \label{alg:lift:matrices} \;
 }
}
\tcp{Solve linear system}
$H \la$ a~$d_0 \times d_W$ matrix over~$\Z_q/p^e$ such that~$R+\sum_i \sum_j h_{i,j} r_{i,j}=0$\; \label{alg:lift:syslin}
\Return~$\widetilde W + p^e V_0 H$\;
\caption{Lift}
\label{alg:lift}
\end{algorithm}

\pagebreak

\begin{proof}
We begin by finding a lift $\widetilde W$ of $W$ to $\Z_q/p^{2e}$ which still represents a subspace of $V$; this goal is attained at line~\ref{alg:lift:Wlift}.

We are then looking for a matrix~$H'$ of the same size as~$W$ and with coefficients in~$p^e \Z_q / p^{2e} \Z_q$ such that~$\widetilde W' = \widetilde W+H'$ satisfies criterion~\eqref{eqn:lift_criterion}. In particular this lift~$\widetilde W'$ must also represent a subspace of~$V$, so we actually look for~$H'$ of the form~$H' = M_V H$ for some matrix~$H$ with coefficients in~$p^e \Z_q / p^{2e} \Z_q$.

Given a matrix~$M$ with~$n_Z$ rows, write~$M_I$ for the matrix formed by the rows of~$M$ indexed by~$I$, where~$I$ is defined at line \ref{alg:lift:I}. The matrices~$\widetilde W_I$ and~$\widetilde W'_I$ are invertible over~$\Z_q/p^{2e}$ by definition of~$I$. Therefore there exists an invertible matrix~$P$ (which is congruent to the identity mod~$p^e$) such that~$\widetilde W'_I P = \widetilde W_I$. Since the matrices~$\widetilde W' P$ and~$\widetilde{W'}$ represent two bases of the \emph{same} subspace of~$V$ ($P$ being the change-of-basis matrix), we may impose that the rows of~$\widetilde W$ indexed by~$I$ remain unchanged, i.e. that~$H'$ be actually of the form~$V_0 H$ for some~$H$. This reduces the number of unknowns and thus improves the efficiency of the algorithm. Note that the number of columns of~$V_0$ is~$\dim V - \dim W = \codim_V W = d_0$.

We then construct the matrix~$K=K(\widetilde W)$, and split it into blocks~$A,B,C,D$ as explained above. In view of \eqref{eqn:lift_criterion}, we have~$D-CA^{-1}B=0$ at least mod~$p^e$, so the division defining~$R$ is exact.

Next, in the double loop we determine for each~$i$ and~$j$ how deforming~$\widetilde W$ by adding~$p^e v_i$ to its~$j$-th column perturbs~$R$ thanks to \eqref{eqn:deform_K}: in each case, we compute~$k$ such that~$K$ shifts by~$p^e k$, and deduce~$r_{i,j}$ such that~$R$ shifts by~$r_{i,j}$.

It then only remains to solve a linear system over~$\Z_q/p^e$ so as to find a linear combination of the~$r_{i,j}$ killing~$R$, and to perform the corresponding change on~$\widetilde W$.

The complexity of this algorithm is dominated by these last two steps. The matrices~$A$ and~$B$ have size~$O(g)\times O(g)$, whereas~$C$ and~$D$ have size~$O(g^2) \times O(g)$, so the matrix computations performed at line \ref{alg:lift:matrices} require  $O(g^{\omega+1})$ operations in~$\Z_q/p^{e}$. In view of remark~\ref{rk:dim_tg} below, the complexity of this algorithm is thus~$O(g^{\omega+3})$ operations in~$\Z_q/p^{O(e)}$.
\end{proof}

\begin{rk}
The matrices~$E_i$ and~$F_m$ are precomputed for efficiency reasons. Besides, multiplications by matrices of the type~$\Delta_c$ should not be performed literally but rather by rescaling the rows or columns as appropriate.
\end{rk}

\begin{rk}\label{rk:dim_tg}
The tangent space of~$J$ has dimension~$g$, and the fibers of the map
\[ \begin{array}{ccc} \Eff^{d_0}(C) & \longrightarrow & J \\ D & \longmapsto & [D-D_0] \end{array} \]
have dimension~$d_0-g$. Since we have rigidified the situation by taking~$H'$ of the form~$V_0 H$, the linear system that we solve at line \ref{alg:lift:syslin} has a solution space of dimension~$d_0=O(g)$. But the matrices $K(\widetilde W)$ and $R$ have size $O(g^2) \times O(g)$, so this system has $O(g^3)$ equations in $O(g^2)$ unknowns. It is thus very redundant, and a lot of time can be saved by solving~$O(g^2)$ random combinations of its equations and checking that we have indeed obtained a solution (e.g. thanks to algorithm~\ref{alg:membership}), instead of solving it directly. Without this, the complexity would be higher than~$O(g^{\omega+3})$.
\end{rk}

\begin{rk}
Even though the performance of our implementation is quite satisfactory (cf. section \ref{sect:examples}), it is unfortunate that the complexity of our lifting algorithm is~$O(g^{\omega+3})$ whereas Makdisi's algorithms perform arithmetic in~$J$ in only~$O(g^\omega)$. In future work, we plan to improve the complexity of our lifting algorithm, by using the asymptotically faster membership test described in Proposition/Algorithm 4.12 of~\cite{Mak2}.
\end{rk}

\section{$p$-adic lifting of torsion points }\label{sect:torslift}

Let~$m \in \N$, and let~$W$ be a~$n_Z \times d_W$ matrix with coefficients in~$\F_q$ representing an~$m$-torsion point~$x \in J(\F_q)[m]$. Thanks to algorithm \ref{alg:lift} presented in the previous section, we can find a lift~$\widetilde W$ of~$W$ representing a lift~$\widetilde x \in J(\Q_q)$ to arbitrary~$p$-adic accuracy; however this point~$\widetilde x$ will not in general be~$m$-torsion. The purpose of this section is to explain how to modify algorithm \ref{alg:lift} so that the output~$\widetilde W$ represents a point of~$J(\Q_q)$ which is still~$m$-torsion.

In this section and in the next one, whenever~$n \in \Z$ and~$W$ is a matrix representing a point~$x \in J$, we will denote by~$[n]W$ a matrix representing~$[n]x \in J$ obtained by repeated use of algorithm \ref{alg:addflip}.  

For simplicity, we will assume that~$m$ is coprime to~$p$; indeed the discussion about formal groups below shows that this ensures that the~$m$-torsion lift of~$x$ to~$J(\Q_q)$ is unique. In practice, we will take~$m =\ell$, a prime distinct from~$p$.

We actually present two methods: the first one is more efficient when the~$p$-adic accuracy is still low, whereas the second one should be used when $p$-adic accuracy becomes higher.

\subsection{Method 1: Killing the kernel of reduction}

Let~$\p = p \Z_q$ be the maximal ideal of~$\Z_q$, and for each~$e \in \N$, let us denote by~$J_e$ the kernel of the reduction map~$\xymatrix{J(\Q_q) \ar@{->>}[r] & J(\Z_q/p^e)}$. For instance~$J_1$ is the kernel of~$\xymatrix{J(\Q_q) \ar@{->>}[r] & J(\F_q)}$.

Since~$J$ has good reduction at~$p$, and since~$\Q_q$ is unramified, the theory of formal groups (cf. for instance~\cite[theorem IV.6.4.b]{Silverman}) provides us with a~$p$-adic Lie~group isomorphism
\[ \log : J_1 \overset{\sim}{\longrightarrow} \p^g \]
which maps~$J_e$ to~$(\p^e)^g$ for all~$e$ (where $(\p^e)^g$ denotes the $g$-fold cartesian product of the set $\p^e$). In particular,~$J_{e}/J_{2e} \simeq (\p^e / \p^{2e})^g$ is an Abelian group of  exponent~$p^e$. 

Let now~$x \in J(\Z_q/p^e)[m]$ be an~$m$-torsion point known to accuracy~$O(p^e)$. It has a unique~$m$-torsion lift~$\widetilde x \in J(\Z_q/p^{2e})[m]$, and its other lifts are of the form~$\widetilde x' = \widetilde x + y$ with~$y \in J_e/J_{2e}$. Therefore, if~$N \in \N$ is such that~$p^e \mid N$ and that~$N \equiv 1 \bmod m$, then for any lift~$\widetilde x'$ of~$x$,~$[N] \widetilde x' = \widetilde x$ is the~$m$-torsion lift of~$x$. This leads to the following algorithm:

\medskip

\begin{algorithm}[H]
\KwIn{A matrix~$W$ over~$\F_q$ representing a point~$x \in J(\F_q)[m]$, and an integer~$e_0 \in \N$.}
\KwOut{A matrix~$\widetilde W$ over~$\Z_q/p^{e_0}$ representing the lift of~$x$ in~$J(\Z_q/p^{e_0})[m]$.}
$\widetilde W \la W$\;
$e \la 1$\;
\While{$e<e_0$}{
$\widetilde W \la$ a lift of~$\widetilde W$ to $\Z_q/p^{2e}$ \tcp*{Using algorithm \ref{alg:lift}}
$N \la$ an integer such that~$p^e \mid N$ and~$N \equiv 1 \bmod m$ \tcp*{By CRT} \label{alg:lifttors_mul_N}
$\widetilde W \la [N] \widetilde W$\;
$e \la 2e$\;
}
\Return~$\widetilde W$\;
\caption{Lifting torsion points using multiplications}
\label{alg:lifttors_mul}
\end{algorithm}

\medskip

The complexity of each iteration of the loop is~$O\big(g^{\omega+3}+g^\omega \log(m p^{e_0}) \big)$ operations in~$\Z_q/p^{O(e_0)}$, the terms coming respectively from algorithm \ref{alg:lift} and from multiplication by~$N=O(mp^{e_0})$.

\begin{rk}
We could also apply algorithm \ref{alg:lift} repeatedly so as to get a lift~$\widetilde W$ of~$W$ to precision~$e_0$, and then multiply by~$N$ where~$p^{e_0-1} \mid N$ and~$N \equiv 1 \bmod m$. However, this would be less efficient as we would perform all the multiplications at high accuracy. Anyway, it is better to use algorithm \ref{alg:lifttors_chart} for large~$e$. 
\end{rk}

\begin{rk}\label{rk:lifttors_mul_noCRT}
In practice, we often use this algorithm to lift an~$\Fl$-basis of~$T_\rho \subset J[\ell]$ from~$J(\F_q)[\ell]$ to~$J(\Z_q/p^{e_0})[\ell]$. In this case, we can afford to multiply each element of the basis by a nonzero scalar in~$\Flx$ in the process, so we can save a bit of effort by suppressing the requirement that~$N \equiv 1 \bmod \ell$ and simply taking~$N = p^e$ at line~\ref{alg:lifttors_mul_N}.
\end{rk}

\subsection{Method 2: Using a coordinate chart at the origin}

We can use algorithm \ref{alg:eval} as a coordinate chart $c' : U \longrightarrow \PP(V)$, where $U \subset J(\Z_q)$ is the preimage of a Zariski-dense subset of $J(\F_q)$. In the rest of this section, we assume that $U$ contains $0 \in J(\Z_q)$, since this is the case for ``most'' choices of divisors $E_1$ and $E_2$ used as parameters by algorithm \ref{alg:eval}. Therefore, we are able to take a lift~$\widetilde W$ of~$W$ thanks to algorithm \ref{alg:lift}, multiply the corresponding point by~$m$, and apply~$c'$ to see where in the vicinity of the origin of~$J$ the result lies. 

Ideally, we would like to compute the differential of the composition of the multiplication-by-$m$ map followed by~$c'$. Unfortunately this seems too complicated, so we settle for the following simpler approach: as explained in remark \ref{rk:dim_tg}, a point  $x \in J(\Z_q/p^e)$ possesses many different lifts $\widetilde x \in J(\Z_q/p^e)$ owing to the tangent space of $J$, so we can take sufficiently many such lifts $\widetilde x_i$, determine the coordinates of the $[m] \widetilde x_i$ thanks to~$c'$, and solve a linear system to deduce a ``combination'' $\widetilde x$ of the $\widetilde x_i$ such that $[m] \widetilde x = 0$.

This idea translates into the following algorithm:

\medskip

\begin{algorithm}[H]
\KwIn{A matrix~$W$ over~$\F_q$ representing a point~$x \in J(\F_q)[m]$, and an integer~$e_0 \in \N$.}
\KwOut{A lift~$\widetilde W$ of~$W$ over~$\Z_q/p^{e_0}$ representing the lift of~$x$ in~$J(\Q_q)[m]$.}
$\widetilde W \la W$\;
$e \la 1$\;
$c'_0 \la$ a column vector with entries in~$\Z_q/p^{e_0}$ representing~$c'(0)$\;
$j_0 \la$ a integer~$\le n_Z$ such that~$c'_0[j_0] \not \equiv 0 \bmod p$\;
$c_0 \la$ dehomogenization of~$c'_0$ wrt. the~$j_0$-th entry\; \label{alg:lifttors_chart_c0}
\While{$e<e_0$}{
\For{$i \la 1$ \KwTo~$g+1$}{ \label{alg:lifttors_mul:parallel}
$\widetilde W_i \la$ lift of~$\widetilde W$ to $\Z_q/p^{2e}$\tcp*{Obtained from algorithm \ref{alg:lift} by picking a random solution to the linear system at line~\ref{alg:lift:syslin}} \label{alg:lifttors_chart:lifts}
$x_i \la$ point of~$J(\Z_q/p^{2e})$ represented by $[m] \widetilde W_i$\;
$c'_i \la$ column vector with entries in~$\Z_q/p^{2e}$ representing~$c'(x_i)$\;
$c_i \la$ dehomogenization of~$c'_i$ wrt. the~$j_0$-th entry\;
$\kappa_i \la (c_i-c_0)/p^e \bmod p^e$\; \label{alg:lifttors_chart:ci}
}
Find~$\lambda_1, \cdots, \lambda_{g+1} \in \Z_q/p^e$ such that~$\sum_{i=1}^{g+1} \lambda_i \kappa_i = 0$ and~$\sum_{i=1}^{g+1} \lambda_i = 1$, or start over with other lifts~$\widetilde W_i$ if no such~$\lambda_i$ exist\; \label{alg:lifttors_chart:ker}
Lift the~$\lambda_i$ to~$\Z_q/p^{2e}$ so that~$\sum_i \lambda_i=1$ still\;
$\widetilde W \la \sum \lambda_i \widetilde W_{i}$\; \label{alg:lifttors_chart:Wtilde}
Check that the point of~$J(\Z_q/p^{2e})$ represented by~$\widetilde W$ is~$m$-torsion, and if not start over with another chart $c'$\; \label{alg:lifttors_chart_testtors}
$e \la 2e$\;
}
\Return~$\widetilde W$\;
\caption{Lifting torsion points using a chart}
\label{alg:lifttors_chart}
\end{algorithm}

\begin{proof}
Let us assume for now that the reduction mod $p$ of the differential of the ``chart''~$c' : U \longrightarrow \PP(V)$ at $0 \in J(\Z_q)$ has full rank $g$, which is extremely likely. Let $c : U \dashrightarrow \Q_q^{n_Z-1}$ denote $c'$ followed by the embedding $\PP(V) \hookrightarrow \PP^{n_Z}$ induced by $V \hookrightarrow \Q_q^{n_Z}$ (cf. section \ref{sect:Mak}) and then dehomogenized wrt. the~$j_0$-th coordinate. Note that by our choice of $j_0$, $c$ is defined on the points of~$J(\Z_q)$ that reduce to~$0 \in J(\F_q)$, and actually sends them to $\Z_q^{n_Z-1}$. Thus for example at line \ref{alg:lifttors_chart_c0}, we compute $c_0 = c(0) \bmod p^{e_0}$.

Suppose now that we have a matrix $W$ with coefficients in $\Z_q/p^e$ for some $e \in \N$ representing a point $ x \in J(\Z_q/p^e)[m]$. Say that a lift of $W$ to $\Z_q/p^{2e}$ is \emph{acceptable} if it represents a point of $J(\Z_q/p^{2e})$, i.e.\ if it represents a subspace of~$V$ and passes the membership test \ref{alg:membership} (but it need not be $m$-torsion), and consider the map
\[ \begin{array}{cccc} \kappa : & \big\{ \ \text{Acceptable lifts of } W \ \big\} & \longrightarrow & (\Z_q/p^e)^{n_Z-1} \\ & \widetilde W & \longmapsto & \displaystyle \frac{c(x_{[m] \widetilde W})-c_0}{p^e}. \end{array} \]
where~$x_{[m] \widetilde W} \in J(\Z_q/p^{2e})$ denotes the point represented by the matrix~$[m] \widetilde W$.
Since~$W$ represents an $m$-torsion point of $J(\Z_q/p^e)$, $x_{[m] \widetilde W}$ reduces to $0 \in J(\Z_q/p^e)$ for all $\widetilde W$, so the expression $c(x_{[m] \widetilde W})$ makes sense and is congruent to $c_0$ mod $p^e$. Thus $\kappa$ is well-defined, and can be used to measure how a lift $\widetilde W$ of $W$ fails to be $m$-torsion, in that
\begin{equation}
\text{The point of } J(\Z_q/p^{2e}) \text{ represented by } \widetilde W \text{ is } m\text{-torsion} \Longleftrightarrow \kappa(\widetilde W) = 0 \bmod p^e. \label{eqn:chart_tors}
\end{equation}

More precisely, let~$\widetilde W$ be an (unknown) fixed acceptable lift of~$W$ representing an~$m$-torsion point of $J(\Z_q/p^{2e})$, so that $\kappa(\widetilde W) = 0 \bmod p^{2e}$, and let~$V_0$ be the matrix defined at line \ref{alg:lift:V0} of algorithm \ref{alg:lift}. 
This algorithm shows that any lift of~$W$ of the form~$\widetilde W + p^e V_0 H$, where~$H$ is matrix corresponding to a vector in the kernel of the linear system solved at line \ref{alg:lift:syslin} of algorithm \ref{alg:lift}, is also an acceptable lift; and for such a lift, we have~$\kappa(\widetilde W + p^e V_0 H) = \kappa(\widetilde W) + T(V_0 H) = T(V_0 H)$, where~$T$ is the differential of $\kappa$, which is the reduction mod $p^e$ of a $\Z_q$-linear map of rank~$g$ by our assumption on~$c'$.

Besides, the lifts constructed by algorithm \ref{alg:lift} all have this form, so the $g+1$ lifts~$\widetilde W_i$ considered at line \ref{alg:lifttors_chart:lifts} of algorithm \ref{alg:lifttors_chart} may be written as $\widetilde W_i = \widetilde W + p^e V_0 H_i$. The~$g+1$ vectors~$\kappa_i = \kappa(\widetilde W_i) = T(V_0 H_i)$ found at line \ref{alg:lifttors_chart:ci} are necessarily linearly dependent, so at line \ref{alg:lifttors_chart:ker} we can find scalars~$\lambda_i \in \Z_q/p^e$ such that~$\sum_{i=1}^{g+1} \lambda_i \kappa_i = 0 \bmod p^e$ non-trivially. This relation can be rewritten as~$\sum_{i=1}^{g} \lambda_i (\kappa_i-\kappa_{g+1}) + \left( \sum_{i=1}^{g+1} \lambda_i \right) \kappa_{g+1} = 0$, so assuming that the $g$ vectors $(\kappa_i-\kappa_{g+1})$ are linearly independent mod $p$ (which happens with high probability, else we start over), $\sum_{i=1}^{g+1} \lambda_i$ cannot be zero mod $p$; we can therefore rescale the $\lambda_i$ so that~$\sum_{i=1}^{g+1} \lambda_i = 1$. Define then
\begin{align*} \widetilde W' &= \sum_{i=1}^{g+1} \lambda_i \widetilde W_i \\ & = \widetilde W + \sum_{i=1}^{g+1} \lambda_i (\widetilde W_i - \widetilde W) \\ &= \widetilde W + p^e V_0 \sum_{i=1}^{g+1} \lambda_i H_i \end{align*}
as at line \ref{alg:lifttors_chart:Wtilde}. This is an acceptable lift of~$W$, which satisfies
\[ \kappa(\widetilde W') = T \left(V_0 \sum_{i=1}^{g+1} \lambda_i H_i \right) = \sum_{i=1}^{g+1} \lambda_i \kappa_i = 0 \bmod p^e;\]
it is therefore~$m$-torsion by \eqref{eqn:chart_tors}.

Now in all generality, as noted in remark \ref{rk:chart}, we cannot be completely certain that the reduction mod $p$ of the differential of $c'$ at~$0$ has full rank $g$. If it is does not, then $c'$ will not be a local immersion, so \eqref{eqn:chart_tors} is only a necessary but not sufficient condition. This nasty case does happen in practice, albeit very rarely, whence the test at line \ref{alg:lifttors_chart_testtors}. In practice, this situation is revealed by the the fact that the vectors~$(\lambda_i)_{i \le g+1}$ satisfying~$\sum_i \lambda_i \kappa_i = 0$ at line \ref{alg:lifttors_chart:ker} form a space of dimension more than 1, so that it is sufficient to execute this test only at the first iteration of the \textbf{while} loop.

The complexity of each iteration of this loop is~$O( g^{\omega+3} + g^{\omega+1} \log m)$ operations in~$\Z_q/p^{O(e_0)}$.
\end{proof}

\begin{rk}
One should not be surprised that the (a priori distinct) matrices $\widetilde W$ and $\widetilde W'$ both represent the unique $m$-torsion lift of $x$ to $J(\Z_q/p^{2e})$, cf. remark \ref{rk:dim_tg}.
\end{rk}

\begin{rk}
Of course, the computation of the $g+1$ lifts $\widetilde W_i$ at line \ref{alg:lifttors_chart:lifts} should not be done by calling algorithm \ref{alg:lift} $g+1$ times, but by modifying it so that it returns $g+1$ random acceptable lifts.
\end{rk}

\begin{rk}\label{rk:lifttors_mix}
When lifting from~$J(\Z_q/p^e)[m]$ to~$J(\Z_q/p^{2e})[m]$, algorithm~\ref{alg:lifttors_chart} performs~$O\big(g^{\omega+3} + g^\omega \log(mp^e) \big)$ operations in~$\Z_q/p^{O(e)}$, whereas algorithm \ref{alg:lifttors_mul} performs~$O(g^{\omega+3} + g^{\omega+1} \log m)$. In order to lift a point from~$J(\F_q)[m]$ to~$J(\Q_q)[m]$, it is therefore reasonable to start with the method of algorithm \ref{alg:lifttors_mul} for low values of~$e$, and to switch to the method of algorithm \ref{alg:lifttors_chart} when~$e$ exceeds~$g \log m / \log p$. If we use fast arithmetic, this results in a complexity of~$\widetilde O(e a \log p) O(g^{\omega+3} + g^{\omega+1} \log m)$ bit operations to lift a point of~$J(\F_q)[m]$ to~$J(\Z_q/p^e)[m]$, where~$q = p^a$. Note however that the loop at line \ref{alg:lifttors_mul:parallel} of algorithm \ref{alg:lifttors_chart} can very well be executed in parallel.
\end{rk}

\section{Application to the computation of Galois representations}\label{sect:galrep}

We now finally come back to the initial problem, that of computing the Galois representation~$\rho$ afforded by the subspace~$T_\rho \subset J[\ell]$. Recall that we denote by~$\Phi$ the Frobenius at~$p$, and that we assume that we know the characteristic polynomial~$\chi_\rho(x) \in \Fl[x]$ of~$\rho(\Phi)$, as well as~$L_p(x) \in \Z[x]$, that of~$\Phi$ acting on~$J$. Also recall that we assume that~$\chi_\rho(x)$ and its cofactor~$L_p(x)/\chi_\rho(x)$ are coprime in~$\Fl[x]$.

\subsection{Splitting the representation}

The first thing to do is to determine~$q=p^a$ so that the points of~$T_\rho$ are defined over~$J(\Q_q)$. The smallest such~$a$ is of course the multiplicative order of~$\rho(\Phi)$, which can be bounded thanks to the following result:

%

\begin{pro}\label{prop:mordroot}
Suppose~$\chi_\rho(x)$ factors in~$\Fl[x]$ as~$\prod_i \chi_i(x)^{e_i}$, where the~$\chi_i(x) \in \Fl[x]$ are pairwise distinct irreducible polynomials of respective degrees~$d_i$, and the~$e_i$ are positive integers. Let~$a_i$ be the multiplicative order of the class of~$x$ in~$\Fl[x]/\chi_i(x)$. Then the order of~$\rho(\Phi)$ is of the form
\[ a = \ell^u \lcm(a_i) \]
for some integer~$u \le \max_i \lceil \frac{\log e_i}{\log \ell} \rceil$, where~$\lceil \cdot \rceil$ denotes rounding from above to the nearest integer.
\end{pro}

\begin{proof}
Let~$\phi = \rho(\Phi)$. Since the~$\chi_i(x)^{e_i}$ are pairwise coprime, the~$\Fl[\phi]$-module~$\overline{T}_\rho$ decomposes as a direct sum~$\bigoplus_i M_i$ of ``generalised eigenspaces''~$M_i = \ker \chi_i(\phi)^{e_i}$. The order of~$\phi$ is thus the lcm of the order of the~$\phi_i$, where~$\phi_i = \phi_{\vert M_i}$. As~$a_i$ divides~$\ell^{d_i}-1$ for all~$i$, the~$a_i$ are coprime to~$\ell$, so it is enough to show that for all~$i$, the order of~$\phi_i$ is of the form~$\ell^{u_i} a_i$ for some integer~$u_i \le \lceil \frac{\log e_i}{\log \ell} \rceil$.  

The characteristic polynomial of~$\phi_i$ is~$\chi_i(x)^{e_i}$, so its eigenvalues in~$\overline \Fl$ are the roots of~$\chi_i(x)$, and are thus of multiplicative order~$a_i$. Triangularising~$\phi_i$ over~$\overline \Fl$ thus shows that~$\phi_i^n$ is unipotent iff.~$a_i$ divides~$n$. In particular,~$a_i$ divides the order of~$\phi_i$.

Write~$\phi_i^{a_i}=1+N_i$ with~$N_i$ nilpotent. The characteristic polynomial of~$N_i$ is~$x^{d_i}$, so~$N^{d_i}=0$ by Cayley-Hamilton. Induction on~$n$ reveals that~$(1+N)^{\ell^n} = 1 + N^{\ell^n}$ for all~$n \in \N$. So if~$U_i = \lceil \frac{\log e_i}{\log \ell} \rceil$, so that~$\ell^{U_i} \ge d_i$, then we have~$\phi_i^{\ell^{U_i} a_i} = (1+N)^{\ell^{U_i}} = 1$, which shows that the order of~$\phi_i$ divides~$\ell^{U_i} a_i$. Since this order is also divisible by~$a_i$, the result follows. 
\end{proof}

\begin{rk}
It is clear that the bound~$u \le \max_i \lceil \frac{\log e_i}{\log \ell} \rceil$ is optimal.
\end{rk}

In order to determine~$a_i$, we simply have to find the primes~$r$ dividing~$n_i = \ell^{d_i}-1$, and to test for each such~$r$ whether~$x^{n_i/r}$ is~$1$ in~$\Fl[x]/\chi_i(x)$.
Fortunately, in practice~$\ell^{d_i}-1$ is not very large, so finding all these primes is not difficult.

In particular, in the case when~$\chi_\rho(x)$ is squarefree, that is to say when~$e_i=1$ for all~$i$, then~$u=0$, so we find that the order of~$\rho(\Phi)$ is exactly~$a=\lcm(a_i)$.

\begin{rk}\label{rk:choose_p}
If~$\chi_\rho(x)$ is not squarefree, then we obtain an upper bound on~$a$ instead. The exact value of~$u$, and thus of~$a$, can then be determined by experimenting, but it is better to avoid this case by changing~$p$ if possible.

In fact, if we know (or can afford to determine)~$\chi_\rho(x)$ for not just one but several primes~$p$, it is advisable for the efficiency of the rest of the computation to choose~$p$ so that~$a$ is minimal; cf. the examples in section \ref{sect:examples}.
\end{rk}

\subsection{Computing a basis of~$T_\rho$ mod~$p$}

Now that we have chosen~$q=p^a$, we can let~$\overline {T_\rho} \subset J(\F_q)[\ell]$ be the reduction of~$T_\rho$ mod~$p$. The next step in the computation of~$\rho$ consists in finding~$\dim \overline{T_\rho}$ matrices~$W$ over~$\F_q$ representing an~$\Fl$-basis of~$\overline{T_\rho}$. Since~$\dim \overline{T_\rho}$ is known (for instance it is the degree of~$\chi_\rho(x)$), our strategy consists in generating elements of~$\overline{T_\rho}$ at random until we obtain a basis, which can be confirmed by algorithm \ref{alg:FR}.

Let~$N = \# J(\F_q)$, which can be computed explicitly as
\[ N = \Res\big(L_p(x),x^a-1\big). \]
Let us factor
\[ N = \ell^v M \]
where~$M \in \N$ is coprime to~$\ell$; we also define
\[ \psi(x) = L_p(x) / \chi_\rho(x) \in \Fl[x]. \]

A simple-minded way to generate elements of~$\overline{T_\rho}$ proceeds as follows:

\medskip

\begin{algorithm}[H]
$x \la$ a random point of~$J(\F_q)$\; 
$x \la [M] x$ \tcp*{Project onto~$J[\ell^\infty]$} \label{Vbad_M}
\While{$[\ell]x \neq 0$}{ \label{Vbad_loop} 
$x \la [\ell] x$ \tcp*{Project onto~$J[\ell]$}
}
$x \la \psi(\Phi) x$ \tcp*{Project onto~$\overline{T_\rho}$} \label{Vbad_psi}
\Return~$x$\;
\caption{An unbalanced way to generate points of~$\overline{T_\rho}$}
\label{alg:Vbad}
\end{algorithm}

\medskip

Since~$N = O(q^g)$ by the Weil bounds, the complexity of this algorithm is~$O\big( (g \log q + a \log \ell) g^\omega \big)$ operations in~$\F_q$.

\bigskip

Unfortunately, the points obtained this way are usually far from being equidistributed. For instance, suppose that~$J(\F_q)[\ell^\infty]$ is of the form~$\Z/\ell^2 \Z \times \Z/\ell \Z$; then with high probability, after line \ref{Vbad_M},~$x$ has the form~$(u,v)$ with~$u$ of order~$\ell^2$ and~$v$ of order~$\ell$, so that the loop at line \ref{Vbad_loop} turns it into~$(\ell u,0)$, so we almost always get points of~$J(\F_q)[\ell] \simeq \Z/\ell \Z \times \Z/\ell \Z$ whose second component is zero.

We can try to circumvent this issue by performing line \ref{Vbad_psi} before the loop at line \ref{Vbad_loop}, but this does not help in case~$\overline{T_\rho}$ itself has~$\ell$-power torsion; besides,~$\psi(x)$, which is computed mod~$\ell$ only, will not be the correct cofactor, so the point we get in the end will not even lie in~$\overline{T_\rho}$ in general!

\bigskip

In order to remedy this issue,~\cite[section 3.9]{Bruin} suggests using the Kummer map
\[ \begin{array}{ccl} J & \longrightarrow & J[\ell] \\ x & \longmapsto & y^\Phi - y, \text{ where } [\ell]y=x. \end{array} \]
Bruin shows that this process leads to uniform distribution in~$J[\ell]$. However, this method requires enlarging the field~$\F_q$, which considerably slows dow the computations, especially the multiplication-by-$M$ step (note that~$M$ is of the order of magnitude of~$q^g$).

\bigskip

We propose instead to modify algorithm \ref{alg:Vbad} by including a process resembling Gaussian elimination. To demonstrate our idea, suppose again that~$J(\F_q)[\ell^\infty] \simeq \Z/\ell^2 \Z \times \Z/\ell \Z$, and take~$T_\rho = J[\ell]$ to simplify. If we start with two random points~$x_1, x_2 \in J$ and we apply algorithm \ref{alg:Vbad}, after line \ref{Vbad_M} these points become~$x_1 = (u_1,v_1)$ and~$x_2 = (u_2,v_2) \in J[\ell^\infty]$, so unless one of the~$u_i$ is 0 mod~$\ell$, line \ref{Vbad_loop} leads us to~$y_1=(\ell u_1,0)$ and~$y_2=(\ell u_2,0)$, which fail to form a basis of~$J[\ell]$. Nevertheless, we can turn this failure to our advantage! Indeed, thanks to algorithm \ref{alg:FR} we can find~$\lambda_1, \lambda_2$ not both 0 such that~$\lambda_1 y_1 + \lambda_2 y_2 = 0$, so dividing this relation by~$\ell$ yields a new~$\ell$-torsion point~$x_3 =  \lambda_1 x_1 + \lambda_2 x_2$, whose second component is nonzero with high probability.

\begin{rk}
Variants of this Gaussian elimination strategy appear in several places of the literature, such as~\cite{Gaussian_elim}.
\end{rk}

This idea leads to the following algorithm, which performs very well in practice:

\medskip

\begin{algorithm}[H]
\KwIn{$\#J(\F_q)=\ell^v M$,~$L_p(x) \in \Z[x]$, and~$\chi_\rho(x) \in \Fl[x]$.}
\KwOut{An~$\Fl$-basis of~$\overline{T_\rho}$.}
$\widetilde \psi(x) \la$ Hensel lift to precision~$O(\ell^v)$ of~$L_p(x)/\chi_\rho(x) \in \Fl[x]$\; \label{alg:Vbasis:Hensel}
$r \la 0$\;
$d \la \deg \chi_\rho(x)$\tcp*{$d = \dim \overline{T_\rho}$}
\While{$r < d$}{
$r \la r+1$\;
 \Repeat{$y \neq 0$}{
 $x \la$ a random point of~$J(\F_q)$\; \label{alg:Vbasis:x}
 $x \la [M] \widetilde \psi(\Phi) x$\;
 $y \la x$\; \label{alg:Vbasis:y}
 $o \la 0$\;
  \While{$[\ell]y \neq 0$}{ 
  $y \la [\ell] y$\;
  $o \la o+1$\;
  }
 }
$x_r \la x$\;
$y_r \la y$\;
$o_r \la o$ \label{alg:Vbasis:or} \tcp*{New point~$y_r = \ell^{o_r} x_r \in \overline{T_\rho}$, can we keep it?}
 \If{$r>1$}{
 $n \la r$\;
 $z_1, \cdots, z_n\la$ random elements of~$J(\F_q)$\; \label{alg:Vbasis:Zi}
 $P \la$ matrix of size~$n \times r$ with coefficients~$P_{i,j} = [ y_j, z_i ]_\ell \in \Fl$\;
 $R \la \Ker P$\;
  \If{$\dim R \ge 2$}{
  $n \la n+1$\;
   \Goto \ref{alg:Vbasis:Zi}\;
  }
  \If{$\dim R = 1$}{
  $(\lambda_1,\cdots,\lambda_r) \la$ a generator of~$R$\;
   \tcp{Is this an actual relation?}
   \If{$\sum_{i \le r} [\lambda_i] y_i = 0$}{ \label{alg:Vbasis:testrel}
    $I \la \{ i \, \vert \, \lambda_i \neq 0 \}$\;
    $\mu \la \min_{i \in I} o_i$\;
    \uIf{$\mu>1$}{
    $x \la \sum_{i \in I} [\lambda_i \ell^{o_i-\mu}] x_i$ \tcp*{Divide the relation}
     \Goto \ref{alg:Vbasis:y}\;
    }
    \Else{
     $r \la r-1$ \tcp*{If cannot divide, give up this point} \label{alg:Vbasis:giveup}
    }
   }
  }
 }
}
\Return~$y_1, \cdots, y_d$\;
\caption{Finding a basis of~$\overline{T_\rho}$}
\label{alg:Vbasis}
\end{algorithm}

\pagebreak

We now show the correctness of this algorithm and discuss its efficiency. Let us begin by introducing the main ideas.

\bigskip

\mypara \label{par:randpt} First of all, we note that this algorithm needs to pick random points of~$J(\F_q)$ at line \ref{alg:Vbasis:x}, and also at line \ref{alg:Vbasis:Zi}. Sophisticated methods to achieve this uniformly are presented in~\cite{Bruin}, but we content ourselves with a much cruder approach: we pick a random subset~$S$ of~$\{1, \cdots, n_Z\}$ of cardinality~$d_0$, and compute by linear algebra the subspace~$W$ of~$V$ formed of vectors whose~$i$-th coordinate vanishes for each~$i \in S$. Indeed, this subspace represents the point~$\left[ \sum_{i \in S} P_i - D_0 \right]$ of~$J$, \linebreak where the~$P_i$ are as in section \ref{sect:Mak}. Since these~$P_i$ have been chosen at random, this approach performs very well in practice. In order to analyse the performance of algorithm \ref{alg:Vbasis}, we will assume that the points thus obtained are uniformly distributed in~$J(\F_q)$.

\smallskip

\mypara \label{par:proj} Let~$L_p(x) = \widetilde \chi_\rho(x) \widetilde \psi(x)$ be the~$\ell$-adic lift to accuracy~$O(\ell^v)$ of the coprime factorisation~$L_p(x) \equiv \chi_\rho(x) \psi(x) \bmod \ell$, and let~$\Lambda = J(\F_q)[\ell^v] \leqslant J(\F_q)$ \linebreak and~$\Gamma=  J(\F_q)[\ell^v,\widetilde \chi_\rho(\Phi)] = \Lambda[\widetilde \chi_\rho(\Phi)] \leqslant \Lambda$. In view of the Chinese remainder isomorphisms
\[ J(\F_q) \simeq J(\F_q)[\ell^v] \times J(\F_q)[M] = \Lambda \times J(\F_q)[M] \]
and
\[ \Lambda \simeq \Lambda[\chi_\rho(\Phi)] \times \Lambda[\psi(\Phi)] = \Gamma \times \Lambda[\psi(\Phi)], \]
if~$g_1,\cdots,g_r \in J(\F_q)$ are uniformly chosen random points as in \ref{par:randpt}, then the points~$h_j=[M] \widetilde \psi(\Phi) g_j$ are uniformly distributed in~$\Gamma$.

\smallskip

\mypara \label{par:prob_gen} Observe now that~$\Gamma$ is a finite Abelian~$\ell$-group, whose~$\ell$-torsion subgroup~$\Gamma[\ell]$ is precisely the space~$\overline{T_\rho}$ that we are trying to generate. It is clear that sufficiently many elements~$h_i$ of~$\Gamma$ generated as in \ref{par:proj} will form a generating set of~$\Gamma$ with high probability; indeed, if~$H= \langle h_1, \cdots, h_k \rangle$ is a strict subgroup of~$\Gamma$, then it has index at least~$\ell$, so~$h_{k+1} \not \in H$ with probability at least~$1-\frac1\ell$.

\smallskip

\mypara \label{par:Gauss} Suppose thus that we have elements~$h_j \in \Gamma$ which form a generating set. We can determine the order of each~$h_j$ by repeatedly multiplying then by~$\ell$ until the result is~$0$. Let~$n \in \N$ be such that the largest of these orders is~$\ell^n$, so that the exponent of~$\Gamma$ is~$\ell^n$, and consider the projection morphism
\[ \pi = [\ell^{n-1}] : \Gamma \longrightarrow \Gamma[\ell] = \overline{T_\rho}. \]
Let~$P$ be the matrix such that~$P_{i,j} = [\pi(h_j),z_i]_\ell \in \Fl$, where the~$z_i$ are sufficiently many random elements of~$J(\F_q)$ generated as in \ref{par:randpt}, and let~$R$ be a matrix whose columns form an~$\Fl$-basis of~$\Ker P$. Then the columns of~$R$ span a space containing all the~$\Fl$-linear relations satisfied by the~$\pi(h_j)$. Besides, for the same reason as in~\ref{par:prob_gen}, the linear forms~$y \mapsto [y,z_i]_\ell$ generate~$\Hom_{\Fl}(\pi(\Gamma),\Fl)$ with high probability; we can check that this indeed the case by testing whether~$\sum_j R_{j,k} \pi(h_j) = 0$ for each~$k$. If this is not the case, we generate more~$z_i$, and try again; else the~$h'_k = \sum_j \widetilde{R_{j,k}} h_j$ form a generating set of~$\Ker \pi = \Gamma[\ell^{n-1}]$, where~$\widetilde{R_{j,k}}$ denotes an arbitrary lift to~$\Z$ of~$R_{j,k} \in \Fl$.

By iterating the process with~$\Gamma[\ell^{n-1}]$ instead of~$\Gamma$ and the~$h'_k$ instead of the~$h_j$, we thus get generating sets for~$\Gamma[\ell^{n-m}]$ for~$m=1,2,\dots$ until we get a generating set for~$\Gamma[\ell^1] = \overline{T_\rho}$.

\pagebreak

\mypara Let now~$h_j \in \Gamma$ be sufficiently many elements, so that they probably generate~$\Gamma$ although we are not sure of that. By applying the method \ref{par:Gauss} to the~$h_j$ (which amounts to working in the subgroup generated by the~$h_j$ instead of~$\Gamma$), either we find a generating set of~$\overline{T_\rho}$, in which case we have reached our goal, so we stop; or we do not, in which case we conclude from \ref{par:Gauss} that our~$h_j$ did not actually generate~$\Gamma$, so we generate a few more of them as in \ref{par:proj} and we try again.

\bigskip

These arguments show that this method quickly leads us to a generating set of~$\overline{T_\rho}$ with high probability. Algorithm \ref{alg:Vbasis} is derived from these arguments, with a few extra optimizations. Namely, given one~$h_j \in \Gamma$, instead of blindly computing~$\pi(h_j) = [\ell^{n-1} h_j]$, which is always~$\ell$-torsion but could be 0, we compute~$[\ell^o] h_j$ where~$o$ is the largest integer such that the result is non-zero; this provides us with more non-zero elements of~$\overline{T_\rho}$, which help to obtain a generating set sooner. Besides, by introducing the~$z_i$ one by one, and by throwing away the ones that yield a linear form which is in the span of the forms~$[ \cdot, z_i]_\ell$ corresponding to the previous~$z_i$, we ensure that these forms remain linearly independent.

In detail, at line \ref{alg:Vbasis:x}, we have obtained~$r-1$ linearly independent points~$y_1, \cdots, y_{r-1}$ of~$\overline{T_\rho}$, with~$y_i = \ell^{o_i} x_i$ and~$x_i \in \Gamma$. After line \ref{alg:Vbasis:or}, we have obtained a new nonzero point~$y_r = \ell^{o_r} x_r \in \overline{T_\rho}$; however this point may be linearly dependent with~$y_1, \cdots, y_{r-1}$, except of course if~$r=1$. In order to determine whether this is the case, we examine these points through the linear forms~$[ \cdot, z_i ]_\ell$  for~$n$ random points~$z_i \in J$, where initially~$n=r$. The space~$R$ we obtain contains the space~$R'$ of linear relations actually satisfied by~$y_1, \cdots y_r$, but may be larger. Since~$y_1, \cdots, y_{r-1}$ are linearly independent,~$R'$ has dimension at most 1; so if~$\dim R \ge 2$ then~$R \supsetneq R'$, so we increase the number~$n$ of linear forms and start our search for relations over. \linebreak If~$\dim R = 1$, then either this candidate relation is a false positive, so the condition at line \ref{alg:Vbasis:testrel} is not satisfied, and~$y_1, \cdots y_r$ are independent so can include~$y_r$ in our partial basis of~$\overline{T_\rho}$; or the candidate relation does hold, but then we can still try to get a new~$\ell$-power torsion point by dividing the relation~$R$ by a power of~$\ell$. We take the result of this division as a new~$x$ and start the whole process over, except if we cannot divide~$R$ by~$\ell$ because one of the~$o_i$ is too small, in which case we give up the current~$x$ and~$y$ at line \ref{alg:Vbasis:giveup} and start over with a new random~$x \in \Gamma$.

\bigskip

The complexity of this algorithm is~$O\big( (g \log q + (a+d) \log \ell) d g^\omega \big)$ where~$d=\dim \rho$, assuming we have to give up the current~$x$ a bounded number of times, and neglecting the cost of discrete logarithms in~$\mu_\ell$ and of linear algebra over~$\Fl$.

\begin{rk} We can further improve the efficiency of the algorithm if, every time a new~$y_r$ has been verified to be linearly independent of~$y_1, \cdots, y_{r-1}$, we repeatedly apply the Frobenius~$\Phi$ to~$y_r$ and add the corresponding points to our partial basis of~$\overline{T_\rho}$ as long as no linear dependence is detected. Indeed, the application of~$\Phi$ by algorithm \ref{alg:Frob} is very fast. This idea is very similar to that presented in section \ref{sect:use_Frob} below.
\end{rk}

\begin{rk}
In order to be able to use the ``linearized'' version of the Frey-R\"uck pairing, we need~$\F_q$ to contain the~$\ell$-th roots of unity (else this pairing would be trivial anyway). This is often the case, since the determinant of the representation~$\rho$ usually involves the mod~$\ell$ cyclotomic character due to the existence of the Weil pairing. In any case we can always extend~$\F_q$, but this would considerably slow down all the computations.
\end{rk}

\subsection{The rest of the computation}

Now that we have obtained an~$\Fl$-basis of~$\overline{T_\rho} \subset J(\F_q)[\ell]$ by algorithm \ref{alg:Vbasis}, we first fix an~$e \in \N$, and we lift our basis into a basis of~$T_\rho \subset J(\Q_q)[\ell]$ to accuracy~$O(p^e)$ by a combination of algorithms \ref{alg:lifttors_mul} and \ref{alg:lifttors_chart} as explained in remarks \ref{rk:lifttors_mix} and \ref{rk:lifttors_mul_noCRT}. We then use algorithm \ref{alg:addflip} to compute all the points of~$T_\rho$, and we evaluate these points by algorithm \ref{alg:eval}, still at precision~$O(p^e)$. All these computations can be massively parallelized, and their complexity is~$\widetilde O(e a \log p) O(g^{\omega+3}+g^{\omega+1} \log \ell + g^\omega \ell^d)$ bit operations.

Finally, we form the monic polynomial~$F(x)$ whose roots are the values in~$\Q_q$ that we have just obtained. Since~$T_\rho$ is globally invariant under Galois,~$F(x)$ has coefficients in~$\Q_p$, and these coefficients approximate rational numbers within~$O(p^e)$. We can identify these rationals by rational reconstruction, provided that~$e$ is large enough.

\begin{rk}
The value of~$e$ is chosen from experience; see section \ref{sect:examples} below for concrete examples. If~$e$ is too low, rational reconstruction will most likely fail, but might also produce incorrect results. Therefore the output of our algorithm to compute~$\rho$ is not guaranteed to be correct, and should be certified by methods such as~\cite{certif}.

Arakelov theory could in principle provide us with bounds on the height of these rationals, but in practice these bounds are often extremely pessimistic, as noted in~\cite{algo}.
\end{rk}

\subsection{Saving effort thanks to the Frobenius}\label{sect:use_Frob}

We have already noted in remark \ref{rk:choose_p} that for the efficiency of the computations, we should if possible choose the prime~$p$ so that the degree~$a = [ \F_q:\F_p]$ is small. However, when this is not possible (e.g. because we cannot afford to compute the local factor~$L_p(x)$ for more than a few small~$p$), we can use the fact that~$a$ is large to our advantage.

\bigskip

Indeed,~$a$ is the order of Frobenius~$\Phi \in \Gal(\Q_q/\Q_p)$. Since the evaluation algorithm \ref{alg:eval} is (purposely!) Galois-equivariant, it commutes with~$\Phi$. This means that when we want to apply it to all the points of~$T_\rho$, we can apply it to only one point in each orbit of~$\Phi$ of~$T_\rho$, and recover the value of the other points simply by applying~$\Phi$ to the result, which is instantaneous. In particular, we only need to \emph{compute} one point of~$T_\rho$ in each orbit instead of all of them, which saves an enormous amount of time, as the orbits of~$\Phi$ on~$T_\rho$ typically have length~$a$.

Similarly, instead of lifting~$p$-adically a whole~$\Fl$-basis of~$T_\rho$ by the method presented in section \ref{sect:torslift}, we can lift a mere~$\Fl[\Phi]$-generating set, after what we can use algorithm \ref{alg:Frob} to recover an~$\Fl$-basis at high~$p$-adic accuracy. The endomorphism induced by~$\Phi$ on~$T_\rho$ is often cyclic; when this is the case, we have to lift only one point instead of~$\dim T_{\rho}$.

\bigskip

\begin{rk}
Although these ideas partly negate the inconvenience of working with a~$p$ such that~$a$ is large, choosing~$p$ so as to minimize~$a$ is still always a good idea. Indeed, the length of an orbit under~$\Phi$ is at most~$a$, so the idea presented above divides the amount of work by at most~$a$, whereas the cost of computing in~$\F_q$ and~$\Q_q$ is super-linear in~$a$.
\end{rk}

\pagebreak

\section{Examples}\label{sect:examples}

We have implemented the algorithms presented in this paper in the C language using the~\cite{gp} library, and we have applied them to a few examples. In most cases, the data we obtained could have been obtained in a much more direct way, but our goal here is to illustrate our methods.

The times given in this section are all CPU times (as opposed to wall times) on the author's laptop. The code is available on the author's personal web page~\cite{webpage} and as a Github repository~\cite{Github}.


\subsection{A modular representation with particularly little ramification}

Let~$\rho : \Gal(\overline \Q / \Q) \longrightarrow \GL_2(\F_7)$ be either of the two mod 7 representations attached to the newform of weight 2 and level~$\Gamma_1(13)$. In~\cite[section 6.3]{companion} we noted that these two representations are twists of each other, and that the corresponding projective representation~$\pi : \Gal(\overline \Q / \Q) \longrightarrow \PGL_2(\F_7)$ corresponds to a particularly lightly ramified Galois number field of Galois group~$\PGL_2(\F_7)$.

As a first demonstration of the methods presented in this article, we are going to compute~$\rho$ and~$\pi$. We thus take~$\ell=7$ and~$C$ the modular curve~$C=X_1(13)$, whose Jacobian~$J$ is such that~$J[7]$ contains an~$\F_7$-subspace~$T_\rho$ of dimension 2 affording~$\rho$. Note that although~$J$ splits up to isogeny into the product of two copies of an elliptic curve defined over the real subfield of the 13th cyclotomic field, this decomposition does not occur over~$\Q$, so computing~$\rho$ through~$J$ is reasonable.

The~\cite{LMFDB} informs us that~$C$ is hyperelliptic of genus~$g=2$ and admits the minimal model
\[ C: y^2+(x^3+x+1)y = x^5+x^4. \]
We use this equation as input data; the fact that~$C$ is hyperelliptic and is a modular curve is completely ignored by our algorithm.

The~\cite{LMFDB} also reveals that the only (up to Galois)  newform of weight 2 and level 13 has coefficients in~$\Q(\sqrt{-3})$, and provides us a few of these coefficients. By reducing these coefficients modulo one of the primes of~$\Q(\sqrt{-3})$ above 7, we get for each prime~$p \not \in \{7,13\}$ the characteristic polynomial of the image of the Frobenius at~$p$ by~$\rho$. Thanks to proposition \ref{prop:mordroot}, we can therefore look for a~$p$ such that the points of~$T_\rho$ are defined over~$\Q_{p^a}$ for a small~$a \in \N$.  

We thus find that for~$p=2$ we must take~$a=8$ or~$24$, depending on which of the primes above~$7$ we consider; for~$p=3$ we need~$a=48$ or~$16$ respectively, and so on. We spot that for~$p=17$ we can take~$a=6$ for either prime above 7, so we choose to take~$p=17$ and~$a=6$; the respective characteristic polynomials being~$x^2-2x-1 \bmod 7$ and~$x^2-x-2 \bmod 7$. Choosing for instance~$\chi_\rho(x)= x^2-2x-1 \bmod 7$ (the other case being completely similar), we can now compute mod~$p^e$ a polynomial~$F(x)$ of degree~$7^2-1$ defining~$\rho$ for arbitrarily large~$e \in \N$.

We take~$e = 32$. The initialization of Makdisi's algorithms to compute in~$J$ at this accuracy takes about 50ms, the computation of an~$\F_7$-basis of~$T_\rho$ mod~$p$ takes 3.5s, lifting this basis to precision~$O(17^{32})$ takes 5.3s, computing~$\F_7$-linear combinations representing all the orbits of~$T_\rho$ under the Frobenius takes 1.8s, and evaluating a map~$J \dashrightarrow \A_1$ at these points takes 4.9s. We then form the polynomial whose roots are these values, and manage to identify its coefficients as rationals (experiments show that this would have failed with~$e=16$ instead of~$32$). All of this takes less than 10ms, and we obtain the polynomial

\tiny
\begin{align*}
F(x) =&  \ x^{48} - \frac{190}{11}x^{47} + \frac{11150}{121}x^{46} - \frac{70770}{121}x^{45} + \frac{192146}{121}x^{44} - \frac{865490}{121}x^{43} + 15942x^{42} - \frac{3378266}{121}x^{41} + \frac{7718056}{121}x^{40} - \frac{7242790}{121}x^{39}\\ + & \frac{161576962}{121}x^{38} - \frac{327843166}{121}x^{37} + \frac{1686874038}{121}x^{36} - \frac{2144483118}{121}x^{35} - \frac{3729825654}{121}x^{34} + \frac{18688301594}{121}x^{33} \\ - & \frac{23348747263}{121}x^{32} - \frac{183640041124}{121}x^{31} + \frac{334672016908}{121}x^{30} - \frac{1036146028660}{121}x^{29} + \frac{124260679100}{11}x^{28} + \frac{3580808205324}{121}x^{27}\\ - & \frac{32666349081700}{121}x^{26} + \frac{64676203524796}{121}x^{25} + \frac{94830437905560}{121}x^{24} - \frac{423358531106188}{121}x^{23} + \frac{261293629597764}{121}x^{22} \\ + & \frac{1780617009288708}{121}x^{21} - \frac{3884461192426292}{121}x^{20} + \frac{3993080217494308}{121}x^{19} - \frac{2118273836414700}{121}x^{18} + \frac{2868502387705524}{121}x^{17} \\+ & \frac{2992809907202955}{121}x^{16} - \frac{8017240457300370}{121}x^{15} + \frac{1032724415961478}{121}x^{14} + \frac{17530455675901702}{121}x^{13} + \frac{938376918522746}{121}x^{12} \\ - & \frac{21242403800528794}{121}x^{11} - \frac{6765741375874194}{121}x^{10} + \frac{812183325256186}{11}x^{9} + \frac{17495913080481536}{121}x^{8} - \frac{5797799442355694}{121}x^{7} \\ & - \frac{14434584737735526}{121}x^{6} + \frac{1659062818893114}{121}x^{5} + \frac{5000613835008606}{121}x^{4} - \frac{1428209615195030}{121}x^{3} - \frac{15369117321210}{11}x^{2} \\ & + \frac{76223729308434}{121}x + \frac{18369946454903}{77}
\end{align*}
\normalsize
which is irreducible over~$\Q$ and defines a number field of discriminant~$-7^{47} 13^{35}$. This is a strong hint that we have correctly identified the coefficients of~$F(x)$, and this can be confirmed rigorously by the methods presented in~\cite{certif}.

By forming symmetric combinations (in this case, products) of the roots of~$F(x)$ along vector lines in~$T_\rho$, we get the polynomial 

\tiny
\[
G(x) = x^8 + \frac{86}{11}x^7 - \frac{35378}{11}x^6 - 28694x^5 + \frac{26301780}{11}x^4 - \frac{729484894}{11}x^3 + 2233638278x^2 - \frac{410091777286}{11}x + \frac{18369946454903}{77}
\]
\normalsize
which is squarefree and therefore defines the projective representation~$\pi$. Using~\cite{gp}'s function \verb?polredabs? to find a canonical polynomial whose root field is the same as that of~$G(x)$ yields the polynomial
\[ G_{\text{red}}(x) = x^8 - x^7 + 7x^6 + 13x - 13 \]
whose Galois group is indeed~$\PGL_2(\F_{7})$. As predicted in~\cite[section 6.3]{companion}, the root discriminant of the splitting field of this polynomial is~$7^{7/8} 13^{5/6} = 27.269\dots$, which is remarkably low. Besides, half of the coefficients of this polynomial are zero; we had encountered the same mysterious phenomenon in~\cite[section 5.1]{companion}, and we still have no explanation.

\begin{rk}Of course, since~$C$ happens to be a modular curve, we could also have computed these representations by the techniques presented in~\cite{algo}. This would actually have taken longer, mainly because of the computation of the periods of the modular curve to high precision, and also because since these methods work over~$\C$, we do not benefit much form the speedup introduced in section \ref{sect:use_Frob} as complex conjugation has order only 2.\end{rk}

\subsection{The torsion of the Klein quartic}

Since it is overkill to use Makdisi's algorithms with hyperelliptic curves (where Mumford coordinates would be much more efficient), and in order to demonstrate the flexibility of our methods, we now experiment with the Klein quartic, given by the affine plane model
\[ C : x^3y+y^3+x=0. \]
This curve is isomorphic to the modular curve~$X(7)$ over~$\Q$, but again our algorithm completely ignores this fact. It has genus~$g=3$, and according to~\cite[section~2.3]{Elkies} its Jacobian~$J$ is isogenous over~$\overline \Q$ to the product of three copies of the elliptic curve~$X_0(49)$, which has CM by~$\Z\left[ \frac{1+\sqrt{-7}}2 \right]$. However, this decomposition does not occur over~$\Q$, since not all local factors~$L_p(x)$ of its~$L$-function are perfect cubes.
Besides,~$C$ dominates~$X_0(49)$, so the Galois representation afforded by the~$\ell$-torsion of~$X_0(49)$ also occurs in~$J[\ell]$. Let us take~$\ell=5$ for instance. By point counting, we spot that for~$p=19$, the local factor~$L_p(x)$ of~$C$ factors as
\[ L_{19}(x) = (x^2+19)(x^4-19x^2+19^2), \]
so the Galois sub-module~$T$ of~$J[5]$ on which the Frobenius at 19 has characteristic polynomial~$x^2+19$ must be isomorphic to the 5-torsion of~$X_0(49)$. Besides, luckily~$x^2+19 \equiv x^2-1 \bmod 5$, so we can take~$a=2$ only, i.e. the points of this sub-module are defined over~$\Q_{19^2}$. Finding two points of~$J(\F_{19^2})$ forming  \linebreak an~$\F_5$-basis of this sub-module takes 2.3s, lifting this basis to precision~$O(19^{32})$ takes 4.5s, forming~$\F_5$-linear combinations representing all the orbits of~$T$ under the Frobenius takes 1.2s, and evaluating the resulting points takes 3.4s. This 19-adic accuracy is sufficient to identify the polynomial

\tiny
\begin{align*}
F(x) =& \ x^{24} - \frac{307088989}{45285641}x^{23} + \frac{650200367}{45285641}x^{22} + \frac{871861629}{45285641}x^{21} - \frac{8986152651}{45285641}x^{20} + \frac{28369847905}{45285641}x^{19} - \frac{57494206833}{45285641}x^{18} \\ + & \frac{85829376881}{45285641}x^{17} - \frac{99861285810}{45285641}x^{16} + \frac{93389301219}{45285641}x^{15} - \frac{71688483976}{45285641}x^{14} + \frac{46005140069}{45285641}x^{13} - \frac{25177926620}{45285641}x^{12} \\ + & \frac{12024739657}{45285641}x^{11} - \frac{5124720540}{45285641}x^{10} + \frac{1973054346}{45285641}x^{9} - \frac{682322926}{45285641}x^{8} + \frac{207466033}{45285641}x^{7} - \frac{53995200}{45285641}x^{6} + \frac{11728577}{45285641}x^{5} \\ - & \frac{2073323}{45285641}x^{4} + \frac{288262}{45285641}x^{3} - \frac{29761}{45285641}x^{2} + \frac{2049}{45285641}x - \frac{361}{226428205}
\end{align*}
\normalsize
Thanks to~\cite{gp}, it is easy to check that this polynomial is irreducible over~$\Q$ and defines the field of definition of a point of~$X_0(49)[5]$. Of course, such a polynomial could have been obtained much more efficiently by working directly with the elliptic curve~$X_0(49)$ instead of~$C$!

\bigskip

Given a prime~$\ell~$, we can also compute the representation
\[ \rho_\ell : \Gal(\overline \Q / \Q) \longrightarrow \GSp_6(\Fl) \]
afforded by the whole of~$J[\ell]$; however this yields polynomials of degree~$\ell^6-1$, so we limit ourselves to~$\ell \le 3$.

Let us start with~$\ell =2$. By point counting, we find that the local factor at~$p=5$ is~$L_5(x)=x^6+125$, which factors mod 2 as~$(x+1)^2(x^2+x+1)^2$; therefore~$J[2]$ is defined over~$\Q_{5^6}$, so we take~$p=5$ and~$a=6$. Let~$\Phi$ be the Frobenius at~$5$. Getting an~$\F_2$-basis of~$J[2]$ over~$\F_{5^6}$ and finding the matrix of~$\Phi$ with respect to it takes 20s; the rational canonical form of this matrix is the cyclic permutation matrix
\[ \left( \begin{matrix} 0 & 0 & 0 & 0 & 0 & 1 \\ 1 & 0 & 0 & 0 & 0 & 0 \\ 0 & 1 & 0 & 0 & 0 & 0 \\ 0 & 0 & 1 & 0 & 0 & 0 \\ 0 & 0 & 0 & 1 & 0 & 0 \\ 0 & 0 & 0 & 0 & 1 & 0 \end{matrix} \right) \in M_6(\F_2), \]
which shows that the~$\Phi$-module~$J[2]$ is isomorphic to~$\F_2[C_6]$, and incidentally that the isogeny~$J \sim X_0(49)^3$ cannot be defined over~$\Q$. Since~$J[2]$ is a cyclic~$\F_2[\Phi]$-module, we can find a lift of this basis to precision~$O(5^{64})$ by lifting a single well-chosen point and taking the translates of the result under~$\Phi$; this takes~8.4s. Forming~$\F_2$-linear combinations representing all the~$\Phi$-orbits of in~$J[2] \setminus \{0\}$ takes~4s, and evaluating the resulting points takes 11.2s. It then takes less than 0.1s to compute and identify the polynomial 

\tiny
\begin{align*}
F(x) =& \ x^{63} - 161x^{60} - 1197x^{58} - 9177x^{57} - 1298x^{56} + 121520x^{55} + 2898448x^{54} + 2416904x^{53} + 4333952x^{52} - 168699608x^{51} \\ - & 476734440x^{50} - 339930164x^{49} + 446081888x^{48} + 30281870248x^{47} - 4473058548x^{46} + 205353521016x^{45} - 707077852236x^{44} \\ + & 1283224964x^{43} - 4774623916424x^{42} + 2627805825680x^{41} + 9669830893184x^{40} + 38725730084552x^{39} + 84184385317648x^{38} \\ - & 207758524883784x^{37} - 28804601640024x^{36} - 1138175980440698x^{35} + 1501648226684368x^{34} + 413550524113048x^{33} \\ + & 987469333006898x^{32} + 1602100443930840x^{31} - 10647759185665742x^{30} + 16341775658984090x^{29} - 4370878629111660x^{28} \\ - & 20130340141861936x^{27} + 7358201939838288x^{26} - 21804407903519528x^{25} + 57105590161516512x^{24} - 1973641573147048x^{23}\\ - & 80209977632969240x^{22} + 122014915556934988x^{21} + 97865010867835456x^{20} - 13990286302239912x^{19} + 95842501513026044x^{18} \\ + & 107824065162314088x^{17} + 102967780530998516x^{16} + 81363980423822628x^{15} - 46964208347978888x^{14} - 11842872637248016x^{13} \\ + & 34046835258483488x^{12} - 61791525835015592x^{11} - 151364668861661808x^{10} - 87421807751514936x^{9} - 21557279135069608x^{8} \\ - & 71522018862910623x^{7} - 67396074567438960x^{6} - 27838001535856792x^{5} - 13272758366940569x^{4} - 15251476649272248x^{3} \\  - & 2716106249233965x^{2} + 4090304480390807x + 1985141227354766
\end{align*}
\normalsize
which is squarefree and factors over~$\Q$ into one factor of degree 1, one of degree 2, two of degree 3, and nine of degree 6.

 It turns out that these factors all define a subfield of the 7th cyclotomic field~$\Q(\zeta_7)$, so, assuming that we have correctly identified the coefficients of~$F(x)$, this means that the field of definition of~$J[2]$ is~$\Q(\zeta_7)$. This is an Abelian number field, and since 5 is a primitive root mod 7, its Galois group is generated by~$\Phi$, the Frobenius at~$5$. Therefore, the image of~$\rho_2$ is cyclic of order 6 and generated by~$\rho_2(\Phi)$, which is the matrix displayed above.

\bigskip

Finally, we redo the same computation, but with~$\ell=3$. Finding a \linebreak prime~$p \not \in \{3,7 \}$ such that~$J[3]$ is defined over a small extension of~$\Q_p$ is not straightforward since the local factor~$L_p(x)$ is always a cube mod 3. In view of the fields in presence, we try~$p=43$ since it is 1 both mod 3 and mod 7, and indeed a little experimentation shows that~$J[3]$ is defined over~$\Q_{43^4}$.

Getting an~$\F_3$-basis of~$J[3]$ over~$\F_{43^4}$ takes 34s, lifting this basis to precision~$O(43^{256})$ takes 120s, forming~$\F_3$-linear combinations representing the orbits under the Frobenius takes 4m, and evaluating the resulting points takes 9m20s. We obtain a polynomial~$F(x)$ of degree 728 whose coefficients are rationals with numerators of up to 163 digits and always the same 133-digit denominator up to small primes, and which is squarefree and factors over~$\Q$ into one factor of degree 8, four of degree 24, and thirteen of degree 48.

These factors all define subfields of 
a number field~$L$ of degree 48 which is Galois, totally imaginary, and has discriminant~$3^{42} 7^{44}$ (in particular its root discriminant is~$3^{7/8} 7^{11/12} = 15.565\cdots$, which is very close to the record given in table 3 p. 134 of~\cite{Odlyzko}).

Thanks to~\cite{Magma} and~\cite{Gap}, we can determine that its Galois group is a non-Abelian nilpotent group of order 48, isomorphic to the direct product of a cyclic group of order 3 and of the normalizer of a non-split Cartan of~$\GL_2(\F_3)$. This is not surprising, since the image of the mod 3 representation attached to the elliptic curve~$X_0(49)$ is precisely this normalizer. In fact,~$L$ turns out to be the compositum of the real subfield~$\Q(\zeta_7)^+$ of the 7-th cyclotomic field and of the field of definition of the 3-torsion of~$X_0(49)$. Besides, the factor of degree 8 of~$F(x)$ defines the field of definition of a point of~$X_0(49)[3]$, and the factors of degree 24 define the compositum of this field and of ~$\Q(\zeta_7)^+$.

\subsection{A higher genus example}

Finally, in~\cite{H2}, we used the methods presented in this article to compute a~$\GL_3(\F_9)$-valued Galois representation afforded by a~6-dimensional~$\F_3$-subspace of the~3-torsion of the Jacobian of a curve of genus~$g=7$. Unlike the previous computations that took place on the author's laptop, this computation took place on a computing cluster with 32 cores provided by Warwick mathematics institute. It only took about one hour of wall time thanks to parallelisation, which demonstrates that our methods are also suitable to higher genera. The CPU times were as follows: getting a basis of this~6-dimensional subspace of~$J(\F_q)[3]$ (with $q=11^{14}$ in this case) took about~2h, lifting this basis to $p$-adic accuracy~$O(11^{1024})$ took~3.5h, forming linear combinations of these points representing all the~$104$ orbits under the Frobenius at~$p=11$ took~6h, and evaluating these points took~12.5h.

It is interesting to note that even for such a large genus, the step of highest complexity, namely lifting $p$-adically the torsion points by a combination of algorithms~\ref{alg:lift},~\ref{alg:lifttors_mul}, and~\ref{alg:lifttors_chart} whose complexity is $O(g^{\omega+3})$, is actually less time-consuming than the subsequent steps, even though their complexity is only~$O(g^\omega)$. This is not so surprising, since the~$O(g^{\omega+3})$~complexity is only due to the very last lines of algorithm~\ref{alg:lift}.

\bigskip

In view of these encouraging results, we plan to use this new $p$-adic method to try to beat the genus records set in \cite{algo} for Galois representations attached to modular forms in the near future.

\end{document}